\documentclass[reqno]{amsart}
\usepackage{amssymb,amsmath}
\usepackage[english]{babel}
\usepackage{amsfonts}
\usepackage{color}
\usepackage{amsthm}
\usepackage[colorlinks=true,linkcolor=NavyBlue, citecolor=NavyBlue,urlcolor=NavyBlue]{hyperref}
\usepackage{graphicx}
\usepackage{mathtools}
\usepackage[usenames,dvipsnames]{pstricks}
\usepackage{bm}
\usepackage{amsmath}
\usepackage{soul}
\usepackage{cancel}

\usepackage{xypic}
\usepackage[draft]{fixme}

\DeclareMathOperator{\AGL}{AGL}

\DeclareMathOperator{\Sym}{Sym}

\DeclareMathOperator{\F}{\mathbb{F}}

\newcommand\deq{\mathrel{\stackrel{\makebox[0pt]{\mbox{\normalfont\tiny def}}}{=}}}

\newcommand{\Size}[1]{\left\lvert #1 \right\rvert}
\newcommand{\Span}[1]{\left\langle\,#1\,\right\rangle}
\newcommand{\Set}[1]{\left\lbrace #1 \right\rbrace}

\newcommand{\puncture}[2]{\mathop{\vee}\!\left[#1;#2 \right]}
\let\phi\varphi

\theoremstyle{plain}
\newtheorem{theorem}{Theorem}
\newtheorem{lemma}[theorem]{Lemma}
\newtheorem{proposition}[theorem]{Proposition}
\newtheorem{corollary}[theorem]{Corollary}

\theoremstyle{remark}
\newtheorem{remark}{Remark}
\newtheorem{fact}{Fact}
\theoremstyle{definition}
\newtheorem{definition}[theorem]{Definition}
\newtheorem*{definition*}{Definition}

\newtheorem*{notation*}{Notation}

\def\hl#1{{\color{NavyBlue}#1}}

\numberwithin{equation}{section}

\title{Rigid commutators and a normalizer chain}

\title{Rigid commutators and a normalizer chain}
\author[R.~Aragona]{Riccardo Aragona}

\author[R.~Civino]{Roberto Civino}
\author[N.~Gavioli]{Norberto Gavioli}
\author[C.~M.~Scoppola]{Carlo Maria Scoppola}
\address[R.~Aragona, R.~Civino, N.~Gavioli and C.M.~Scoppola]{
	Dipartimento di Ingegneria e Scienze dell'Informazione e Matematica \\ 
	Universit\`a degli Studi dell'Aquila\\
	via Vetoio\\
	I-67100 Coppito (AQ)\\
	Italy}
\email[Riccardo Aragona]{riccardo.aragona@univaq.it}
\email[Roberto Civino]{roberto.civino@univaq.it}
\email[Norberto Gavioli]{norberto.gavioli@univaq.it}
\email[Carlo Maria Scoppola]{carlo.scoppola@univaq.it}

\thanks{All the authors are members of INdAM-GNSAGA (Italy) and of the
  group  ``Crittografia   e  Codici''   of  the   Italian  Matematical
  Union. R.  Civino is  partially funded by  the Centre  of Excellence
  EX-EMERGE at University of L'Aquila.  N. Gavioli is a member of
  the Centre of Excellence EX-EMERGE at University of L'Aquila.  }


\begin{document}

\begin{abstract}
  The novel notion of rigid commutators is introduced to determine the
  sequence of  the logarithms of  the indices of a  certain normalizer
  chain  in the  Sylow $2$-subgroup  of the  symmetric group  on $2^n$
  letters. The  terms of this sequence  are proved to be  those of the
  partial  sums of  the partitions  of an  integer into  at least  two
  distinct parts, that relates to a famous Euler's partition theorem.
\end{abstract}

\maketitle

\section*{Introduction}

In a  recent paper~\cite{Aragona2020},  the authors observed  a rather
surprising  coincidence  between  the sequence of integers
\begin{equation*}
  1,  2, 4,  7, 11,  16,  23, 32,  43, 57\ldots  
\end{equation*}
representing the partial sums of  the famous sequence $\{b_j\}$ of the
number of  partitions of the  integer $j$  into at least  two distinct
parts, already studied by Euler~\cite{euler1748introductio}, and a  sequence
 of  group-theoretical
invariants. Our sequence arises in  connection with  a problem  in algebraic
cryptography, namely the  study of  the conjugacy  classes of 
affine elementary abelian  regular subgroups of the  symmetric group on
$2^n$  letters~\cite{cds06,calderini2017translation,Aragona2019}. This
is relevant  in the cryptanalysis of  block ciphers,
since it may trigger a  variation of the well-known \emph{differential
  attack}~\cite{bih91}: a  statistical attack which allows  to recover
information  on the  secret unknown  key by  detecting a  bias in  the
distribution of the  \emph{differences} on a given  set of ciphertexts
when the  corresponding plaintext difference is  known. In particular,
if $\mathbb F_2^n$ serves as the  message space of a block cipher (see
e.g.~\cite{aes})  which  has  been   proven  secure  with  respect  to
differential   cryptanalysis~\cite{nyberg1995provable}   and  if   $T$
represents the translation group on  $\mathbb F_2^n$, any conjugate of
$T$ can be potentially used to define new alternative operations on $\mathbb F_2^n$ for a
successful            differential           attack~\cite{Civino2019}.
In~\cite{Aragona2020}, on  the basis  of the  aforementioned interest,
the  authors studied  a chain  of normalizers,  which begins  with the
normalizer $N_n^0$ of $T$ in  a suitable Sylow $2$-subgroup $\Sigma_n$
of  $\Sym(2^n)$  and whose  $i$-th  term  $N_n^i$  is defined  as  the
normalizer in  $\Sigma_n$ of the  previous one.  After  providing some
experimental   as   well   as  theoretical   evidence,   the   authors
conjectured~\cite[Conjecture~1]{Aragona2020}         the        number
$\log_{2}\Size{N^{i}_n  : N^{i-1}_n}$  to  be independent  of $n$  for
$1\le  i\le n-2$,  and  to be  equal  to the  $(i+2)$-th  term of  the
sequence of the partial sums  of the sequence $\{b_j\}$\footnote{\ The
  sequence  $b_j+1$ appears  in several  others areas  of mathematics,
  from  number theory  to commutative  algebra~\cite{Enkosky2014}.  In
  particular, it was  already known to Euler  that $b_j+1$ corresponds
  to the number of partitions of $j$ into odd parts (see~\cite[Chapter
  16]{euler1748introductio}  and  \cite[\S 3]{Andrews2007}).   Several
  proofs of  this {Euler's partition  theorem} have been  offered ever
  since~\cite{Syl1882,  Andrews1994, Kim1999},  and several  important
  refinements                         have                        been
  obtained~\cite{Syl1882,Fine1988,Bes1994,Bous97,Straub2016}.}
previously      mentioned~\cite[\url{https://oeis.org/A317910}]{OEIS}.
\newline \newline In this paper  we completely settle this conjecture.
The first  attempts to  solve this problem  were based  on theoretical
techniques   which   clashed   with  their   own    growing
computational complexity.   For this reason,  we develop here  a novel
framework to approach the problem from  a different point of view.  In
this new approach, indeed, we take into account both the imprimitivity
and the nilpotence  of the Sylow $2$-subgroup  $\Sigma_n$ to represent
its elements in terms of  a special family of left-normed commutators,
that   we  call   \emph{rigid  commutators},   in  a   fixed  set   of
generators. Any such commutator $[X]$  can be identified with a subset
$X$  of $\{1,\dots,n\}$.   The  subgroups of  $\Sigma_n$  that can  be
generated  by   rigid  commutators  are called here
\emph{saturated subgroups}.  A careful inspection led us to prove that
the normalizers $N^i_n$ are saturated subgroups.  In particular, a set
of generators of  $N^i_n$ can be obtained from a  set of generators of
$N^{i-1}_n$ by adding the rigid commutators  of the form $[X]$ for all
$X$  such  that the  elements  of  the  complementary  set of  $X$  in
$\{1,\dots,k\}$,  where  $k=\max  X  \le  n$,  yield  a  partition  of
$i+2-n+k$ into at  least two distinct parts. This is  the key to prove
the  conjecture.
\newline
\newline
The advantage  of  adopting  rigid
commutators is twofold.  In the first place, they prove to be handy in
calculations with the use of  the \emph{rigid commutator machinery}, a
dedicated set of rules which we develop in this paper. Secondly, rigid
commutators can  be seen  as factors  in a  \emph{unique factorization
  formula} for  the elements  of any  given saturated  subgroup.  This
representation is crucial in showing  that the normalizers $N^i_n$ are
saturated.  By means of this result and of the machinery, we derive an
algorithm  which efficiently  computes  the normalizer  chain.
\newline
\newline
The paper is  organized as follows: in Section~\ref{sec:prel}
some basic facts  on the Sylow $2$-subgroup  $\Sigma_n$ of $\Sym(2^n)$
are recalled.  Section~\ref{sec:commutators} is totally devoted to the
introduction  and   the  study  of   rigid  commutators  and   to  the
construction    of    the     rigid    commutator    machinery.     In
Section~\ref{sec:main} the rigid commutator machinery is used to prove
the     conjecture    on     the    normalizer     chain    previously
mentioned~\cite[Conjecture~1]{Aragona2020}. 
In Section~\ref{sec:normalizer_saturated} it  is shown that each  term of
the normalizer chain  is a saturated group and  an efficient procedure
to determine the  rigid generators of the normalizers  is derived.  An
explicit construction  of the normalizer  chain in a specific  case is
provided   in    Section~\ref{sec:computation}, and  some  open    problems arising 
from computational evidence   are
discussed.  Finally, some  hints for  future investigations
are presented in Section~\ref{sec:concl}.

\section{The Sylow $2$-subgroup of $\Sym(2^n)$}\label{sec:prel}
Let $n$ be a non-negative integer.  We start recalling some well-known
facts about the Sylow $2$-subgroup
$\Sigma_n$  of  the  symmetric  group on  $2^n$  letters.
\newline
\newline
\noindent Let us consider the set
\begin{equation*}
  \mathcal{T}_n=\bigl\{w_1\dots  w_{n}   \mid  w_i  \in
  \{0,1\} \bigr\}
\end{equation*}
of binary words of length $n$, where $\mathcal{T}_0$ contains only
the empty word. The  infinite rooted binary tree  $\mathcal{T}$ is defined
as the graph  whose vertices are $\bigcup_{j\ge  0} \mathcal{T}_j$ and
where two  vertices, say  $w_1\dots w_{n}$  and $v_1\dots  v_{m}$, are
connected   by    an   edge    if   $|m-n|=1$   and    $w_i=v_i$   for
$1 \leq i \leq  \min(m,n)$. The empty word is the  root of the tree
and it is connected with both the two words of length $1$.

\noindent  We  can  define  a   sequence  $\Set{s_i}_{i  \geq  1}$  of
automorphisms of  this tree.  Each  $s_i$ necessarily fixes  the root,
which is the only vertex of degree $2$. The automorphism $s_1$ changes
the  value $w_1$  of the  first letter  of every  non-empty word  into
$\bar{w}_1\deq  (w_1+1)   \bmod  2$  and  leaves   the  other  letters
unchanged.  If $i\ge 2$, we define
\begin{equation}\label{eq:generators}
  (w_1\dots w_{n})s_i\deq
  \begin{cases}
    \text{empty word} & \text{if $n=0$} \\
    w_1\dots \bar{w_i}\dots  w_{n} & \text{if $n\ge i$ and $w_1=\dots=w_{i-1}=0$}\\
    w_1\dots w_{n} & \text{otherwise.}
  \end{cases}
\end{equation}
In general, $s_i$  leaves a word unchanged unless the  word has length
at least $i$ and the letters preceding the $i$-th one are all zero, in
which  case the  $i$-th letter  is increased  by $1$  modulo $2$.   If
$i \le n$  and the word $w_1\dots w_n\in  \mathcal{T}_n$ is identified
with                            the                            integer
$1+\sum_{i=1}^{n}2^{n-i} w_{i}\in \Set{1,\dots, 2^n}$, then $s_i$ acts
on $ \mathcal{T}_n$ as the  the permutation whose cyclic decomposition
is
\begin{equation*}
  \prod_{j=1}^{2^{n-i}}(j,j+2^{n-i})
\end{equation*}
which has order $2$.   In particular, the group $\Span{s_1,\dots,s_n}$
acts  faithfully on the  set  $\mathcal{T}_n$,  whose cardinality  is
$2^n$,  as a  Sylow  $2$-subgroup $\Sigma_n$  of  the symmetric  group
$\Sym(2^n)$ (see also Fig.~\ref{fig:tree}).

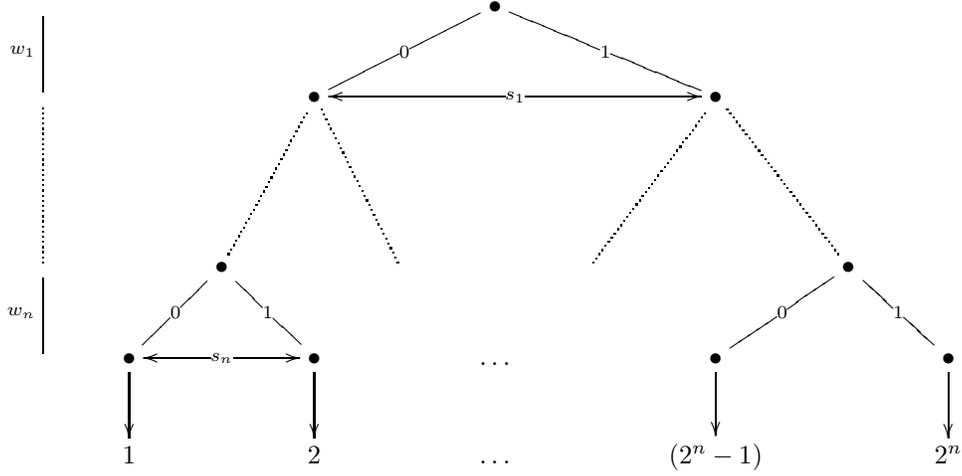
\begin{figure}
  \begin{equation*}
    \xymatrix{
      \ar@{-}[d]_{w_1} & && & &\bullet \ar@{-}[dll]|0 \ar@{-}[drr]|1 && && \\
      \ar@{.}[dd] & && \bullet \ar@{.}[ddl] \ar@{.}[ddr]\ar@{<->}[rrrr] |{s_1}  & && & \bullet \ar@{.}[ddl] \ar@{.}[ddr] & &\\
      &&&&&&&&& \\
      \ar@{-}[d]_{w_{n}} & & \bullet \ar@{-}[dl]|0 \ar@{-}[dr]|1 &&&&&& \bullet \ar@{-}[dl]|0 \ar@{-}[dr]|1 &\\
      & \bullet \ar@{<->}[rr]|{s_n} &  & \bullet &&\dots && \bullet & & \bullet \\
      & 1\ar@{<-}[u] &  & 2\ar@{<-}[u] &&\dots && (2^{n}-1)\ar@{<-}[u] & & 2^n \ar@{<-}[u]
    }
  \end{equation*}
  \caption{\footnotesize  The  action  of $\Sigma_n$  on  the  subtree
    $\bigcup_{i=0}^n \mathcal{T}_i$.}
  \label{fig:tree}
\end{figure}
It is also well known that
\begin{equation*}
  \Sigma_{n}	= \Span{s_n} \wr \Sigma_{n-1}   = \Span{s_n} \wr \dots \wr \Span{s_1} \cong \wr_{i=1}^n C_2 
\end{equation*}  
is the iterated wreath product of $n$ copies of the cyclic group
$C_2$ of order $2$.\newline
\newline
The  \emph{support} of a permutation is the  set of the
letters  which  are  moved  by   the  permutation.  We  say  that  two
permutations  $\sigma$ and  $\tau$  are \emph{disjoint}  if they  have
disjoint  supports; two
disjoint permutations always commute.  

The \emph{closure}
\begin{equation*}
  S_i\deq\Span{s_i}^{\Span{s_1,\dots,s_i}}
\end{equation*}
is  generated by  disjoint  conjugates  of $s_i$,  hence  $S_i$ is  an
elementary abelian $2$-group which is normalized by $S_j$ if $j\le i$.
Moreover,
$\Sigma_n=S_1  \ltimes  \dots  \ltimes S_n\cong  \Sigma_{n-1}  \ltimes
S_n$.

\section{Rigid      commutators}      \label{sec:commutators}      The
\emph{commutator}  of two  elements  $h$ and  $k$ in  a  group $G$  is
defined as $[h,k]\deq h^{-1}k^{-1}hk=h^{-1}h^k$.
The    \emph{left-normed    commutator}    of   the    $m$    elements
$g_1,\dots,g_m\in  G$  is  the  usual   commutator  if  $m=2$  and  is
recursively defined by
\begin{equation*}
  [g_1,\dots,g_{n-1},g_m]\deq \bigl[[g_1,\dots,g_{m-1}],g_m\bigr]
\end{equation*} if
$m\ge 3$.   It is well  known that the  commutator subgroup $G'$  of a
finitely generated nilpotent group $G$ can be generated by left-normed
commutators         involving        only         generators        of
$G$~\cite[III.1.11]{Huppert1967}.
From   now  on,   we  will   focus  on   left-normed  commutators   in
$s_1, \ldots,  s_n$.  For the sake of simplicity,  we
write   $[i_1,\dots,i_k]$  to   denote   the  left-normed   commutator
$[s_{i_1},\dots, s_{i_k}]$, when $k\ge 2$,  and we also write $[i]$ to
denote the element $s_i$.

\begin{definition}\label{def:rigid_commutators}
  A  left-normed commutator  $[i_1,\dots,i_k]$ is  called \emph{rigid,
    based at $i_1$  and hanging from $i_k$},  if $i_1>i_2>\dots >i_k$.
  Given  a subset  $X=\Set{i_1,\dots, i_k}  \subseteq \Set{1,\dots,n}$
  such that $i_1>i_2>\dots >  i_k$, the \emph{rigid commutator indexed
    by  $X$},   denoted  by  $[X]$,  is   the  left-normed  commutator
  $[i_1,\dots,i_k]$.  We set $[X]\deq  1$ when $X=\emptyset$.  The set
  of  all   the  rigid  commutators   of  $\Sigma_n$  is   denoted  by
  $\mathcal{R}$ and 
  we let $\mathcal{R}^*\deq \mathcal{R}\setminus\Set{[\emptyset]}$.
\end{definition}

At the  end of  this section we prove   that  every permutation  in the
Sylow $2$-subgroup $\Sigma_n$ can be expressed,  in a unique way, as a
product  of  the objects  previously  defined.   To this  purpose,  we
develop  below a  set of  rules to  perform computations  with (rigid)
commutators.
\subsection{Rigid commutator machinery}\label{subsec:rcm}
Let $1  \leq i_1,  i_2, \ldots,  i_k \leq  n$ be  integers and  let us
consider the  commutator $[i_1,\dots  ,i_k]$. The following  facts are
easily checked.

\begin{fact}\label{fact1}
  Denoting  by   $i  =   \max\Set{i_1,\dots  ,i_k}$,   the  commutator
  $[i_1,\dots ,i_k]$ is  a product of conjugates of $s_{i}$  by way of
  elements in  $\Span{s_{i_1},\dots ,s_{i_k}}$ and thus  it belongs to
  $S_i$. Any  two such  conjugates commute, since  they belong  to the
  same $S_i$.
\end{fact}
\begin{fact}\label{fact2}
  As     a    direct     consequence    of     Fact~\ref{fact1},    if
  $\max\Set{i_1,\dots    ,i_k}=    \max\Set{j_1,\dots   ,j_l}$    then
  $[i_1,\dots ,i_k]$ and $[j_1,\dots,j_l]$ commute.
\end{fact}

\noindent Note that if $g\in S_i$  and $h\in S_j$, then $[g,h]\in S_k$,
where $k=\max\Set{i,j}$, so  $[g,h]^2=1$ since $S_k$ is elementary
abelian.   It  follows that  $[g,h,h]=[g,h]^2[g,h,h]=[g,h^2]=[g,1]=1$.
As a consequence we have:
\begin{fact}\label{lem:expunge_new}
  If $k \geq 2$ and $i_j = i_{j+1}$ for some $1 \leq j \leq k-1$, then
  $[i_1,\dots, i_k]=1$.
\end{fact}

The following  result is  crucial since  it allows  us to  rewrite every
commutator as a rigid commutator.
\begin{lemma}\label{lem:compound_commutators}
  Let $k\ge 2$ and $l\ge 1$ be integers. If
  \begin{equation*}
    c\deq[[i_1,\dots,i_k],[j_1,\dots,j_l]]
  \end{equation*}
  is the commutator of the two rigid commutators $[i_1,\dots,i_k]$ and
  $[j_1,\dots,j_l]$,  then
  \begin{enumerate}
  \item  \label{item:symmetry}the   order  of  $c$  divides   $2$,  so
    $ c=[[j_1,\dots,j_l],[i_1,\dots,i_k]]$;
  \item \label{item:same_start} if $i_1=j_1$, then $c=1$;
  
  \item \label{item:expunged}if $l\ge  2$ and $i_k >j_l  $, then $j_l$
    can be dropped, i.e.
    \begin{equation*}
      c= [[i_1,\dots,i_k],[j_1,\dots,j_{l-1}]];
    \end{equation*}
  \item        \label{item:disjoint_comm}         if        $i_1>j_1$,
    $\Set{i_1,\dots,i_k}   \cap  \Set{j_1,\dots,j_l}=\emptyset$,   and
    $s\deq\max\Set{h \mid i_h> j_1}$, then
    \begin{equation*}
      c=[i_1,\dots, i_s, j_1];
    \end{equation*}
  \item \label{item:sameend} if $l\ge 2$ and $i_k=j_l$, then
    \begin{equation*}
      c= [[i_1,\dots,i_{k-1}],[j_1,\dots,j_{l-1}], j_l ];
    \end{equation*}
  \item  \label{item:general_commutator} if  $i_s  >j_1 \ge  i_{s+1}$,
     then
    \begin{equation*}
      c=[i_1,\dots,      i_s,      j_1,      h_1,\dots,
      h_t],
    \end{equation*}   
     where $h_1 > \dots > h_t$ and $\Set{h_1,     \dots,     h_t}\deq    \Set{i_1,\dots,i_k}     \cap
    \Set{j_1,\dots,j_l}$.      Moreover,          $c=1$         if
    $j_1\in {\Set{i_1, \dots, i_k}}$.
  \end{enumerate}
\end{lemma}
\begin{proof}
  Let us prove each claim separately.
  \begin{enumerate}
  \item The claim $c^2 = 1$ depends on the fact that $c\in S_i$, where
    the         index         $i$        is         defined         as
    $i\deq\max\Set{i_1,\dots,i_k,j_1,\dots,j_l}$.
  \item   If   $i_1=j_1$,   then   both   of   $[i_1,\dots,i_k]$   and
    $[j_1,\dots,j_l]$ belong  to $S_{i_1}$ which is  abelian, thus the
    claim follows.
  \item Assume that $l\geq 2$ and $j_l < i_k$. In this case
    \begin{align*}
      c&=[i_1,\dots,i_k][i_1,\dots,i_k]^{[j_1,\dots,j_{l-1}]s_{j_l}[j_1,\dots,j_{l-1}]s_{j_l}} \\
       &= [i_1,\dots,i_k]\bigl([i_1,\dots,i_k]^{[j_1,\dots,j_{l-1}]}\bigr)^{s_{j_l}[j_1,\dots,j_{l-1}]s_{j_l}}.
    \end{align*}
    The    permutations     $s_{j_l}[j_1,\dots,j_{l-1}]s_{j_l}$    and
    $[i_1,\dots,i_k]^{[j_1,\dots,j_{l-1}]}$  are  disjoint: the  first
    one           has          support           contained          in
    $\Set{2^{n-j_l}+1,\dots  , 2^{n-j_l+1}}$  and the  support of  the
    second one is contained in
    \begin{equation*}
      \Set{1,\dots, 2^{n-\min(i_k,j_{l-1})+1}}
      \subseteq \Set{1,\dots,
        2^{n-j_l}}.
    \end{equation*}
    Hence
    \begin{equation*}
      c=[i_1,\dots,i_k][i_1,\dots,i_k]^{[j_1,\dots,j_{l-1}]}=[[i_1,\dots,i_k],[j_1,\dots,j_{l-1}]],
    \end{equation*}
    which proves the claim.
  \item     The    claim     follows    by    a repeated applications of items
    \eqref{item:expunged} and \eqref{item:symmetry}.
  \item 
 For every $x,y\in G \deq\Span{s_{n},\dots, s_{i_{l}+1}}$ the permutations $x$ and $y^{s_{j_l}}$  are disjoint and so they commute. In particular, if $x^2=1$, then $[x,s_{j_l}]^2= (xx^{s_{j_l}})^2= x^2(x^2)^{s_{j_l}}=1$. 	If $a,b\in G $ are such that $a^2=b^2=1$, then
	\begin{align*}
		[[a,b],s_{j_l}] &= [abab, s_{j_l}]\\
		 &= ababa^{s_{j_l}} b^{s_{j_l}}a^{s_{j_l}}b^{s_{j_l}}  =   aa^{s_{j_l}}bb^{s_{j_l}}aa^{s_{j_l}} b b^{s_{j_l}} \\
		& = [a,s_{j_l}] [b,s_{j_l}] [a,s_{j_l}] [b,s_{j_l}] \\
		&= [a,s_{j_l}]^{-1} [b,s_{j_l}]^{-1} [a,s_{j_l}] [b,s_{j_l}] =[[a,s_{j_l}],[b,s_{j_l}]].
	\end{align*}
For $a\deq [i_1,\dots,i_{k-1}]$ and $b\deq [j_1,\dots,j_{l-1}]$, we have 
\begin{equation*}
		[[i_1,\dots,i_{k-1},j_l],[j_1,\dots,j_{l-1}, j_l ]]= [[i_1,\dots,i_{k-1}],[j_1,\dots,j_{l-1}], j_l ],
\end{equation*} 
	as required. 
  \item     An     iterated      use     of   items  \eqref{item:symmetry},
    \eqref{item:expunged} and \eqref{item:sameend} yields
    \begin{equation*}
      c=[[i_1,\dots,    i_s],   [j_1,    \dots,   j_v],
      h_1,\dots, h_t]
    \end{equation*} if $j_1 > i_{s+1} \ge h_1$, where
    the                                                   intersection
    $\Set{i_1,\dots,  i_s}\cap  \Set{j_1,  \dots,  j_v}=\emptyset$  is
    trivial, while,        if      $j_1=h_1=i_{s+1}$,      then
    $c=[[[i_1,\dots,     i_s,h_1],      h_1],\dots,     h_t]$.      By
    Fact~\ref{lem:expunge_new},             the             commutator
    $[[i_1,\dots, i_s,h_1],  h_1]$ is trivial,  and so $c=1$.   We may
    then    assume   that    $j_1    >   i_{s+1}    \ge   h_1$.     By
    \eqref{item:disjoint_comm},     we     obtain     the     equality
    $[[i_1,\dots,  i_s], [j_1,  \dots,  j_v]]=[i_1,\dots, i_s,  j_1]$,
    therefore  \begin{equation*}
      c=[i_1,\dots,   i_s,  j_1,  h_1,\dots,
      h_t]
    \end{equation*}
    as claimed. \qedhere
  \end{enumerate}
\end{proof}

A repeated application  of Lemma~\ref{lem:compound_commutators} shows
that every left-normed commutator  $[i_1,\dots,i_k]$ can be written as
a            commutator            $[j_1,\dots,j_l]$,            where
$\Set{j_1,\dots,j_l}         \subseteq\Set{i_1,\dots,i_k}$         and
$j_h \ge  j_{h+1}$ for all  $1\leq h  \leq l-1$. If  $j_h=j_{h+1}$ for
some     $h$,    then     Fact~\ref{lem:expunge_new}    shows     that
$[j_1,\dots,j_h,j_{h+1}]=1$, which in turn implies $[j_1,\dots,j_l]=1$.
This fact is summarized in the following result.

\begin{proposition}\label{prop:no_repetitions}
  Any  left-normed commutator  $[i_1,\dots,i_k]$ can  be written  as a
  rigid   commutator   $[j_1,\dots,j_l]$,   for  a   suitable   subset
  $\Set{j_1,\dots,j_l} \subseteq\Set{i_1,\dots,i_k}$.
\end{proposition}
It is worth noticing here that rigid commutators are the images of
 P.~Hall's basic commutators~\cite{Hall1934} under the presentation
of the group  $\Sigma_n$ as a factor of the  $n$-generated free group,
once the order of the generators is reversed.

\subsection{Saturated subgroups}
In this section we give a representation of the elements of $\Sigma_n$
in terms of rigid commutators. 
\begin{lemma}\label{lem:basis}
  The set of all the rigid commutators $[X]\in \mathcal{R}$, where $X$
  varies among the subsets of $\Set{1,\dots,n}$ such that $\max(X)=i$,
  is a basis for $S_i$.
\end{lemma}
\begin{proof}
  Let $1 \leq i \leq n$. To prove  the claim, we look at
  $S_i$  as   a  $2^{i-1}$-dimensional   vector  space   over  $\F_2$.
  Proceeding by backward  induction on $j$, for $i \geq  j \geq 1$, we
  show  that  the  set  of  all  the   rigid
  commutators based at $i$ and hanging  from $h$, for some {$h\ge j$},
  is linearly independent.  When $j=i$  there is nothing to prove. 
  Assume 
  \begin{equation}\label{eq:claim}
    \prod_{i>i_1>\dots > i_t\ge j}[i,i_1,\dots, i_t]^{e_{i,i_1,\dots, i_t}} = 1,
  \end{equation}
  where the exponents  are in $\F_2$.  We aim at  proving that all the
  exponents are $0$. From Eq.~\eqref{eq:claim} we have
  \begin{equation*}
    \prod_{i>i_1>\dots > i_t> j}[i,i_1,\dots, i_t]^{e_{i,i_1,\dots, i_t}}\prod_{i>i_1>\dots > i_{t-1} > i_t= j}[i,i_1,\dots, i_{t-1}, j]^{e_{i,i_1,\dots,i_{t-1}, j}} = 1,
  \end{equation*}
  and so
  \begin{multline}\label{eq:basis}
    \prod_{i>i_1>\dots > i_t> j}[i,i_1,\dots, i_t]^{e_{i,i_1,\dots, i_t}} = \\
    \left [\left(\prod_{i>i_1>\dots  > i_{t-1} >  i_t= j}[i,i_1,\dots,
        i_{t-1}]^{e_{i,i_1,\dots,i_{t-1}, j}}\right) ,j\right ] .
  \end{multline}
  Note   that   if   the    permutation   on   the   right-hand   side
  of~Eq.~\eqref{eq:basis} is non-trivial, then  it moves some $x$ with
  $x>2^{n-j}$, which is fixed by the one on the left-hand  side. Hence
   the  permutations  on both  sides  are  trivial.   By
  induction,    the   exponents    in    the    left-hand   side    of
  Eq.~\eqref{eq:basis} are all $0$.  Now, the commutator map
  \begin{equation*}
    [\,\cdot ,s_j]\colon \Span{s_{j+1},\dots,s_n} \to
    \Span{s_{j},\dots,s_n}
  \end{equation*}
  is injective, hence the equality
  \begin{equation*}
    \left [\left(\prod_{i>i_1>\dots  > i_{t-1} >  i_t= j}[i,i_1,\dots,
        i_{t-1}]^{e_{i,i_1,\dots,i_{t-1}, j}}\right) ,j\right ] =1
  \end{equation*}
  implies
  \begin{equation*}
    \prod_{i>i_1>\dots    >    i_{t-1}   >    i_t=    j}[i,i_1,\dots,
    i_{t-1}]^{e_{i,i_1,\dots,i_{t-1},  j}}=1.
  \end{equation*}  
  Again,     by     the      inductive     hypothesis,     we     find
  $e_{i,i_1,\dots,i_{t-1},     j}=0$    for     every    choice     of
  $ i_1 >\dots >i_{t-1}$.  As the number of rigid commutators based at
  $i$ equals the dimension of $S_i$, the proof is complete.
\end{proof}

We  can  now  state  our   first  main  result  as  a  straightforward
consequence  of Lemma~\ref{lem:basis}.   Let  us  call a  \emph{proper
  order}  $\prec$  on $\mathcal{R}^*$  any  total  order refining  the
partial order  defined by  $[i_1,\dots,i_t] \prec  [j_1,\dots,j_l]$ if
$i_1 <  j_1$. Here  we denote  by $\mathcal{P}_{n}$  the power  set of
$\Set{1,\dots,n}$.

\begin{theorem}\label{thm:unique_representation}
  Given  a  proper order  $\prec$  in  $\mathcal R^*$,  every  element
  $g\in \Sigma_n$ can be uniquely represented in the form
  \begin{equation*}
    g
    = \prod_{Y\in \mathcal{P}_{n}\setminus \Set{\emptyset}}[Y]^{e_{g}(Y)},
  \end{equation*}
  where  the  factors  are  ordered  with  respect   to  $\prec$  and
  $e_{g} \colon \mathcal{P}_{n}\setminus\Set{\emptyset} \to \Set{0,1}$
  is a function depending on $g$.
\end{theorem}

\begin{proof}
  Since $\Sigma_n  = S_1 \ltimes  \dots \ltimes  S_n$, the claim  is a
  straightforward consequence of Lemma~\ref{lem:basis}.
\end{proof}
Some of the following corollaries  are straightforward and their proof
will be omitted.
\begin{corollary}\label{cor:size_subgroup}
  If $G$  is a  subgroup of $\Sigma_n$  containing $k$  distinct rigid
  commutators, then $\Size{G}\ge 2^k$.
\end{corollary}

We now need a new concept which  plays a key role in the remainder of
this work.
\begin{definition}
  A subset  $\mathcal{G}$ of $\mathcal{R}$ is  called \emph{saturated}
  if  $\mathcal{G}\cup  \Set{[\emptyset]}$   is  closed  under  taking
  commutators and  the subgroup  $G\deq\Span{\mathcal{G}}\le \Sigma_n$
  is called a \emph{saturated subgroup}.
\end{definition}
\begin{remark}\label{rem:sturated_generation}
  A subgroup  $G\le \Sigma_n$ is  saturated if and  only if it  can be
  generated by some subset  $\mathcal{X}$ of $\mathcal{R}$: indeed $G$
  is   also   generated   by   the  smallest   saturated   subset   of
  $\mathcal{G}\cap \mathcal{R}$ containing $\mathcal{X}$.
\end{remark}
\begin{corollary}\label{cor:saturated_subgroups}
  Let $G\le \Sigma_n$ be a saturated subgroup generated by a saturated
  set $\mathcal{G}\subseteq  \mathcal{R}^*$ and let $\prec$ be any given
  proper order on  $\mathcal{G}$. Every element $g\in G$  has a unique
  representation
  \begin{equation*}
    g=\prod_{c\in \mathcal{G}}c^{e_c(g)},
  \end{equation*}
  where the  commutators in  the product are  ordered with  respect to
  $\prec$     and     $e_c(g)\in    \Set{0,1}$.      In     particular
  $\Size{G}=2^{\Size{\mathcal{G}}}$.
\end{corollary}

\begin{corollary}\label{cor:homogeneus_components}
  Let $G\le \Sigma_n$ be a saturated subgroup generated by a saturated
  set $\mathcal{G}\subseteq  \mathcal{R}^*$ and let $\prec$  any given
  proper order on $\mathcal{G}$.  If the product $c_1\cdots c_k\in G$,
  where             $c_i\in            \mathcal{R}^*$,             and
  $c_1\precneqq   c_{2}   \precneqq   \dots   \precneqq   c_k$,   then
  $c_i\in \mathcal{G}$ for all $1 \leq i \leq k$.
\end{corollary}
\begin{proof}
  Note that  since every rigid  commutator belongs to some  $S_i$, the
  group    $G$    has    the    semidirect    product    decomposition
  $G=(G\cap  S_1)\ltimes \dots  \ltimes  (G\cap  S_n)$. In  particular
  every  element of  $G$  can  be written  as  an  ordered product  of
  elements  of  $\mathcal{G}$.   Write $c_1\cdots  c_k=g_1\cdots  g_t$
  where $g_i\in  \mathcal{G}$ and $g_1\precneqq \dots  \precneqq g_t$.
  By  Theorem~\ref{thm:unique_representation},   we  have   $k=t$  and
  $c_i=g_i\in \mathcal{G}$.
\end{proof}

The      next       statement      follows       immediately      from
Corollary~\ref{cor:homogeneus_components}.
\begin{corollary}\label{cor:homogeneus_components2}
  Let $G\le \Sigma_n$ be a  saturated subgroup.  If $g=g_1\cdots g_n$,
  where $g_i\in S_i$ for  $1\leq i \leq n$, then $g\in  G$ if and only
  if  $g_i\in  G\cap  S_i$  for  $1  \leq  i  \leq  n$.   Moreover  if
  $g=h_1\cdots h_n$,  where $h_i\in S_i$ for  $1 \leq i \leq  n$, then
  $h_i=g_i$ for $1 \leq i \leq n$.
\end{corollary}

\section{Elementary  abelian regular  $2$-groups  and  their chain  of
  normalizers}\label{sec:main}

A vector  space $T$ of dimension  $n$ over $\F_2$ acts  regularly over
itself as a group  of translations. By way of this  action, $T$ can be
seen as a regular elementary  abelian subgroup of $\Sym(2^n)$, and any
other regular elementary abelian  subgroup of $\Sym(2^n)$ is conjugate
to  $T$ in  $\Sym(2^n)$~\cite{Dixon1971}.   The normalizer  of $T$  in
$\Sym(2^n)$ is  the affine  group $\AGL(T)$, where  $T$ embeds  as the
normal subgroup of translations.  For this  reason, we refer to any of
the conjugates of $T$ as a \emph{translation subgroup} of $\Sym(2^n)$.
Every  chief  series  $\mathfrak{F}=\Set{T_i}_{i=0}^n$ of  $T$,  where
$1=T_0 <  T_1 <  \dots <  T_n=T$, is normalized  by exactly  one Sylow
$2$-subgroup  $U_{\mathfrak{F}}$ of  $\AGL(T)$.   In \cite[Theorem  p.
226]{Leinen1988} it  is proved that every  chief series $\mathfrak{F}$
of $T$ corresponds to  a Sylow $2$-subgroup $\Sigma_{\mathfrak{F}}$ of
$\Sym(2^n)$ containing $T$  and having a chief  series that intersects
$T$        in        $\mathfrak{F}$.         The        correspondence
$\mathfrak{F}\mapsto \Sigma_\mathfrak{F}$  is a bijection  between the
sets of the chief series of $T$  and the set of Sylow $2$-subgroups of
$\Sym(2^n)$ containing $T$.  In  \cite{Aragona2020} it is also pointed
out                                                               that
$U_{\mathfrak{F}}=N_{\Sigma_{\mathfrak{F}}}(T)=\Sigma_{\mathfrak{F}}\cap
\AGL(T)$. From now  on the chief series $\mathfrak{F}$  will be fixed,
and   so,   without   ambiguity,   we  will   write   $\Sigma_n$   and
$U_n$      to     denote      respectively
$\Sigma_{\mathfrak{F}}$ and $U_{\mathfrak{F}}$.  In \cite{Aragona2019}
it is  proved that  $U_n$ contains, as  normal subgroups,  exactly two
conjugates  of   $T$,  namely  $T$  and   $T_{U_n}=T^{g}$,  {for  some
  $g  \in  \Sym(2^n)$}.    It  is  also  shown   that  the  normalizer
$N_n^1=N_{\Sym(2^n)}(U_n)$  interchanges  by   conjugation  these  two
subgroups and that $N_n^1$ contains  $U_n$ as a subgroup of index~$2$.
In particular, $N_n^1\le \Sigma_n$.  In  the following section we will
extend  these results  on  $T,  U_n, N_n^1$  to  the  entire chain  of
normalizers, which is defined below.
\subsection{The normalizer chain}
The \emph{normalizer chain starting at $T$} is defined as
\begin{equation}\label{eq:normalizer_chain}
  N_n^i \deq \begin{cases}
    U_n = N_{\Sigma_n}(T)& \text{if $i=0$},\\
    N_{\Sigma_n}(N^{i-1}_{n}) & \text{if $i\ge 1$.}
  \end{cases}
\end{equation}
In      \cite{Aragona2020}      the      authors      proved      that
$N_{\Sigma_n}(N_{n}^i)=N_{\Sym(2^n)}(N_{n}^i)$, for all  $i\ge 0$, 
computed the  normalizer chain for $n  \le 11$ by way  of the computer
algebra package \textsf{GAP} \cite{GAP4}, and 
conjectured  that the  index $\Size{N^{i+1}_n  : N_n^{i}}$  does not
depend on $n$ for $n\ge i+3$~\cite[Conjecture 1]{Aragona2020}.  In this
section we
prove this conjecture arguing by induction,  by  means of  the  rigid
commutator machinery developed  in Section~\ref{subsec:rcm}.  We start
by defining
\begin{equation*}
  T\deq\Span{t_1,\dots, t_n}
\end{equation*}
where
\begin{equation*}
  t_i\deq[s_i,s_{i-1},\dots,s_1]=[i,i-1,\dots,1] \in \mathcal{R}^*.
\end{equation*}

\begin{lemma}\label{lem:translation_subgroup}
  $T$ is  an elementary  abelian regular  subgroup of  $\Sigma_n$.  In
  particular, $T$ is a translation subgroup of $\Sym(2^n)$.
\end{lemma}
\begin{proof}
  $T$  is a  subgroup of  $\Sigma_n$ as  it is  generated by  elements
  belonging to  $\Sigma_n$.  By item  \ref{item:general_commutator} of
  Lemma~\ref{lem:compound_commutators} it  follows that $[t_i,t_j]=1$,
  so that $T$ is abelian.  Note that $t_i^2=1$ as $t_i\in S_i$, and so
  $T$ is elementary abelian of order  at most $2^n$.  Let us now prove
  that  $T$ is  transitive.   Let $1  \le  x \le  2^n$  be an  integer
  represented as $x= 1+\sum_{i=1}^{n}2^{n-i} w_{i}$ in binary form and
  let $t=\prod_{i=1}^nt_i^{w_i}$.  A direct check shows that $t$ moves
  $1$ to $x$.  Since  $T$ has an orbit with $2^n$  elements and it has
  order at most  $2^n$, it follows that $\Size{T}=2^n$  and that every
  point stabilizer is trivial, therefore  $T$ is a regular permutation
  group on $\Set{1,\dots, 2^n}$.
\end{proof}

Let us now  determine the permutations in  $\Sigma_n$ normalizing $T$.
For      $1      \leq      j<i\leq     n$      let      us      define
$X_{ij}\deq \Set{1,\dots,i}\setminus \Set{j}$ and
\begin{equation*}
  u_{ij}\deq [X_{ij}] = [i,\dots,j+1,j-1,\dots,1] \in \mathcal R^*.
\end{equation*}
From now on we will set
\begin{equation*}
  \mathcal{U}_n \deq \Set{t_1,\dots,t_n, u_{ij} \mid 1 \leq j < i \leq
    n} \subseteq \mathcal{R}^*.
\end{equation*}
\begin{proposition}\label{prop:U_as_normalizer}
  The  group  $\Span{\mathcal{U}_n}$  is  the  normalizer  of  $T$  in
  $\Sigma_n$, i.e.
  \begin{equation*}
    U_n = \Span{T,u_{ij} \mid 1 \le j < i \le n}.
  \end{equation*}
\end{proposition}
\begin{proof}
  Let us set $U \deq \Span{T,u_{ij} \mid 1 \le j < i \le n}$ and
  let   us   prove   that   $U   =   U_n   =   N_{\Sigma_n}(T)$.    By
  Lemma~\ref{lem:compound_commutators} we have
  \begin{equation*}
    [t_h,u_{ij}]=
    \begin{cases}
      1 & \text{if $h\ne j$} \\
      t_i & \text{if $h=j$}.
    \end{cases}
  \end{equation*}
  This shows that $U\le N_{\Sigma_n}(T) = U_n$ and that $\mathcal U_n$
  is       a        saturated       set.         Therefore,       from
  Corollary~\ref{cor:saturated_subgroups},
  $\Size{U}=   2^{\Size{\mathcal   U_n}}  =   2^{\frac{n(n+1)}{2}}   =
  \Size{U_n}$, which proves the claim.
\end{proof}

We aim at proving our second  main result, providing the generators of
the normalizer $N_n^i$  in terms of rigid commutators.   The result is
proved by induction on~$i\ge 1$.
\subsection*{Induction basis}
Let  us denote  by  $\eta_n$ the  rigid commutator  based  at $n$  and
hanging from $3$ such that no intermediate integer is missing, i.e.
\begin{equation}\label{eq:n11} 
	\eta_n\deq [n,n-1,\dots,3].
\end{equation}
We now prove that we can generate $N_n^1$ by appending $\eta_n$ to the
list $\mathcal{U}_n$ of the rigid commutators generating $U_n$.
\begin{proposition}\label{prop:n1}
  If  $n\ge 3$,  then the  group $\Span{\mathcal{U}_n,\eta_n}$  is the
  normalizer $N_n^1$ of $U_n$ in $\Sigma_n$, i.e.
  \begin{equation*}
    N_n^1 = \Span{T,u_{ij},\eta_n \mid 1 \le j < i \le n}.
  \end{equation*}
  Moreover, $\Size{N_n^1:U_n}=2$.
\end{proposition}

\begin{proof}
  By Lemma~\ref{lem:compound_commutators},
  \begin{equation*}
    [t_i,\eta_n] =
    \begin{cases}
      u_{n,2} & \text{if $i=1$} \\
      t_n & \text{if $i=2$}\\
      1 & \text{otherwise}
    \end{cases}
    \text{\quad and \quad } [u_{ij},\eta_n] =
    \begin{cases}
      u_{n,1}& \text{if $i=2$ and $j=1$}\\
      1 & \text{otherwise}
    \end{cases},
  \end{equation*}
  Thus the  rigid commutator $\eta_n$ belongs  to $N_{\Sigma_n}(U_n)$,
  hence    $\Span{U_n,\eta_n}   \le    N_{\Sigma_n}(U_n)$.    Moreover
  $U_n \cap  S_n=\Span{t_n,u_{n,1},\dots,u_{n,n-1}}$ and  so $\eta_n$,
  which is based at $n$, is  such that $\eta_n \notin U_n$.  The claim
  now       follows        from       $\Size{N_{\Sigma_n}(U_n):U_n}=2$
  \cite[Theorem~7]{Aragona2019}.
\end{proof}

\subsection*{Inductive step}
 Let {$1\leq b  \leq n$} and let
$I$ be a (possibly empty)  subset of $\Set{1,2,\dots,b-1}$.  We define
the \emph{rigid commutator based at $b$ and punctured at $I$} as
\begin{equation}
  \puncture{b}{I} \deq [  \Set{1,\dots,b} \setminus
  I] \in \mathcal{R}^*
\end{equation} 
and, if $I = \Set{i_1,i_2,\dots,i_k}$ we also denote $\puncture{b}{I}$
by $\puncture{b}{i_1,i_2,\dots,i_k}$.\\
For example, the permutation $\eta_n$ defined in Eq.~\eqref{eq:n11} is
equal to $\puncture{n}{2,1}$.  \medskip
 
\noindent We also define
\begin{equation}
  \label{eq:C_lk}
  \mathcal{W}_{ij}\deq \Set{\puncture{i}{I} \in \mathcal{R}^* \,\,\Big  | \,\, I \subseteq
    \Set{1,2,\dots,i-1},  \Size{I}  \ge  2,  \vphantom{\sum}\smash{\sum_{x  \in  I}  x}=j  \;
  }
\end{equation}
for each $1 \leq i \leq n$ and $j$, and
\begin{equation}\label{def:Ni}
  \mathcal{N}_n^{i}\deq
  \begin{cases}
    \mathcal{U}_n & \text{if $i=0$} \\
    \mathcal{N}_n^{i-1}    \dot\cup    \left(    \dot\bigcup_{j=1}^{i}
      \mathcal{W}_{n+j-i,\,j+2} \right) & \text{for $i>0$.}
  \end{cases}
\end{equation}
Note  that, if  $j  \leq i-2$,  then $\Size{  \mathcal{W}_{i,j}}=b_j$,
i.e.\  the number  of partitions  of $j$  into at  least two  distinct
parts.  Our next goal  is to prove that $N_n^i=\Span{\mathcal{N}_n^i}$
for  each  $0\le   i  \le  n-2$,  where  $N_n^i$  is   defined  as  in
Eq.~\eqref{eq:normalizer_chain}. Propositions~\ref{prop:U_as_normalizer}
and \ref{prop:n1} show that this is actually the case when
$i\in \{0,1\}$.   \newline
\newline
In  order to  prove the  general result,  we need  the following  
reformulation      of       item~\ref{item:general_commutator}      of
Lemma~\ref{lem:compound_commutators} to  compute commutators  of rigid
commutators written in punctured form.

\begin{proposition}\label{prop:punctured}
  Let  $1  \leq   a,b  \leq  n$  and  let  $I$   and  $J$  subsets  of
  $\Set{1,2,\dots,a-1}$ and $\Set{1,2,\dots,b-1}$ respectively. Then
  \begin{equation*}
    \bigl[\,
    \puncture{a}{I} ,
    \puncture{b}{J}
    \, \bigr] = 
    \begin{cases}
      \puncture{\max (a,b)}{\;  (I \cup  J)\setminus\Set{\min(a,b)}} &
      \text{if $\min (a,b)\in I \cup J$}
      \\
      1 & \text{otherwise}.
    \end{cases}
  \end{equation*}
\end{proposition}
\begin{proof}
  Let $c\deq \bigl[\, \puncture{a}{I} , \puncture{b}{J} \, \bigr]$. If
  $a=b$, then  $c=1$. Without loss  of generality, we can  assume that
  $a> b$.   By Lemma~\ref{lem:compound_commutators}, if $b  \notin I$,
  then   $c=1$.   If   $b\in   I$,  the   claim   follows  from   item
  \ref{item:general_commutator}                                     of
  Lemma~\ref{lem:compound_commutators}.
\end{proof}

In the following facts, we summarize  some properties that will  be  useful  in 
the proof of the conjecture.

\begin{fact}\label{lem:N_i-1}
  A  commutator $\puncture{a}{J}$  such that  $1  \leq a  \leq n$  and
  $J \subseteq \Set{1,2,\dots,a-1}$  belongs to $\mathcal{N}_n^{i}$ if
  and only if one of the following conditions is satisfied:
  \begin{enumerate}
  \item \label{item:b} $J=\emptyset$, and so $\puncture{a}{J}=t_a$;
  \item \label{item:c}  $\Size{J}=1$, and  so $\puncture{a}{J}=u_{aj}$
    where $J=\Set{j}$;
  \item    \label{item:a}    $\Size{J}\geq     2$, and
    $\sum_{j\in J} j \le i+2-(n-a)$.
  \end{enumerate}
\end{fact}

\begin{fact}\label{rem:reduction_for_N}
  Note     that      for     $2\le      i\le     n-2$      the     set
  $\mathcal{N}_n^i\cap (S_1\ltimes \dots \ltimes S_{n-1})$ is equal to
  $\mathcal{N}_{n-1}^{i-1}$.   Indeed, at  the  $i$-th iteration,  the
  newly generated  elements of  $\mathcal{N}_n^i$, which are  those in
  $\mathcal{N}_n^i\setminus  \mathcal{N}_n^{i-1}$, are  constructed by
  \emph{lifting}              the             elements              of
  $\mathcal{N}_n^{i-1}\setminus    \mathcal{N}_n^{i-2}$,   i.e.\    by
  replacing a rigid commutator based  at $j$ with the rigid commutator
  obtained by removing  its left-most element, for $j \leq  n$, and by
  adding some new  rigid commutators based at $n$,  in accordance with
  Eq.~\eqref{def:Ni}.  Proceeding in this way  it is easy to check that,
  disregarding    all    the    commutators   based    at    $n$    in
  $\mathcal{N}_n^{i}$, the lifted elements are exactly the elements of
  $\mathcal{N}_{n-1}^{i-1}$.     The    reader    is    referred    to
  Section~\ref{sec:computation} for explicit examples.
\end{fact}

\begin{fact}\label{rem:reduction_for_N_v2}
  In   the  proof   of   Proposition~\ref{prop:n1}   we  showed   that
  \(    [\,\mathcal{N}_n^{1}   ,    \mathcal{N}_n^{0}   ]    \subseteq
  \mathcal{N}_n^{0} \cup \Set{[\emptyset]}  \).  Assuming by induction
  on            $2\le             i\le            n-2$            that
  $[\,\mathcal{N}_{n-1}^{i-1},\mathcal{N}_{n-1}^{i-2}]       \subseteq
  \mathcal{N}_{n-1}^{i-2}   {\cup    \Set{[\emptyset]}}$   and   using
  Fact~\ref{rem:reduction_for_N}, we can conclude that
  \begin{multline*}
    [\,\mathcal{N}_n^{i}  \cap  (S_1\ltimes  \dots  \ltimes  S_{n-1}),
    \mathcal{N}_n^{i-1} \cap (S_1\ltimes \dots \ltimes S_{n-1})] =
    \\
    [\,\mathcal{N}_{n-1}^{i-1},   \mathcal{N}_{n-1}^{i-2}]   \subseteq
    \mathcal{N}_{n-1}^{i-2} {\cup \Set{1}}  = \mathcal{N}_n^{i-1} \cap
    (S_1\ltimes \dots \ltimes S_{n-1}) \cup \Set{[\emptyset]}.
  \end{multline*}
  Similarly,
  \begin{equation*}
    [\,\mathcal{N}_n^{i}  \cap  (S_1\ltimes  \dots  \ltimes  S_{n-1}),
    \mathcal{N}_n^{i}   \cap  (S_1\ltimes   \dots  \ltimes   S_{n-1})]
    \subseteq   \mathcal{N}_n^{i}  \cap   (S_1\ltimes  \dots   \ltimes
    S_{n-1}) \cup \Set{[\emptyset]}.
  \end{equation*}
\end{fact}

From the previous fact we have that, in order to prove by induction on
$i$                                                               that
$[\,\mathcal{N}_n^{i},                   \mathcal{N}_n^{i-1}]\subseteq
\mathcal{N}_n^{i-1}      \cup     \Set{[\emptyset]}$      and     that
$[\,\mathcal{N}_n^{i},  \mathcal{N}_n^{i}]\subseteq  \mathcal{N}_n^{i}
\cup     \Set{[\emptyset]}$,    it     suffices    to     show    that
$[\,\mathcal{W}_{n,\,i+2}\,    ,\,   \mathcal{N}_n^{i-1}]    \subseteq
\mathcal{N}_n^{i-1}\cup       \Set{[\emptyset]}$        and       that
$[\,\mathcal{W}_{n,\,i+2}\,    ,\,     \mathcal{N}_n^{i}]    \subseteq
\mathcal{N}_n^{i}\cup \Set{[\emptyset]}$. This  is accomplished in the
following result.

\begin{lemma}\label{lem:rigid_commutators_normalize}
  If                 $i\le                  n-2$,                 then
  $[\mathcal{N}_n^i,\mathcal{N}_n^{i-1}]\subseteq  \mathcal{N}_n^{i-1}
  \cup                      \Set{[\emptyset]}$                     and
  $[\mathcal{N}_n^i,\mathcal{N}_n^i]\subseteq   \mathcal{N}_n^i   \cup
  \Set{[\emptyset]}$.
\end{lemma}
\begin{proof}
  If
  $\puncture{n}{I}\in                   \mathcal{W}_{n,\,i+2}\subseteq
  \mathcal{N}_n^i\setminus           \mathcal{N}_n^{i-1}$          and
  $\puncture{a}{J}      \in     \mathcal{N}_n^{i-1}$      then,     by
  Proposition~\ref{prop:punctured},
  \begin{equation*} c  \deq\bigl[\, \puncture{n}{I}  , \puncture{a}{J}
    \, \bigr] =
    \begin{cases}
      \puncture{n}{\; (I   \cup  J)\setminus\Set{a}} & \text{if   $a\in I$}\\
      1 & \text{otherwise.}
    \end{cases}
  \end{equation*}
  We   may   assume   $a\in   I$.    From   Fact~\ref{lem:N_i-1},   if
  $\puncture{a}{J}$ is as in case \eqref{item:a}, we have
  \begin{eqnarray*}
    \sum_{x\in  (I  \cup  J)\setminus\Set{a}}   x  &\le&  \sum_{x\in  J}x +\sum_{x\in I} x - a \\
                                                   &\le& i+1-(n-a) + i+2-(n-n) -a \\
                                                   &=& i+2 -(n-i-1) \\
                                                   &\le& i+2 -1=i+1,
  \end{eqnarray*}
  and so $c  \in \mathcal{N}_n^{i-1}$.  If $\puncture{a}{J}$  is as in
  case \eqref{item:b}, i.e.\ $\puncture{a}{J}=t_a$, then we have
  \begin{equation*}
    \sum_{x\in  (I \cup  J)\setminus\Set{a}} x=  \sum_{x\in I}  x -  a
    =i+2-a    \le    i+1
  \end{equation*}
  and so, also in this  case, $c\in \mathcal{N}_n^{i-1}$.  Finally, if
  $\puncture{a}{J}$    is   as    in   case    \eqref{item:c},   i.e.\
  $\puncture{a}{J}=u_{a,j}$, we have
  \begin{equation*}
    \sum_{x\in (I \cup  J)\setminus\Set{a}} x \le \sum_{x\in I}  x - a
    +j  = i+2-(a-j)  \le i+1
  \end{equation*}
  and again $c\in \mathcal{N}_n^{i-1}$. Similar computations prove that, if $\puncture{a}{J} \in \mathcal{N}_n^{i}$, then also
  $c \in \mathcal{N}_n^{i}$.
\end{proof}

The following result is now straightforward.
\begin{proposition}\label{prop:size_of_Ni}
  The set  $\mathcal{N}_n^i$ is a  saturated set of  rigid commutators
  and
  $\Span{\mathcal{N}_n^{i}}                                        \le
  N_{\Sigma_n}\bigl(\Span{\mathcal{N}_n^{i-1}}\bigr)$.       Moreover,
  $\Size{\Span{\mathcal{N}_n^{i}}}=2^{\Size{\mathcal{N}_n^i}}$.
\end{proposition}
\begin{proof}
  The claim  follows from Lemma~\ref{lem:rigid_commutators_normalize},
  Fact~\ref{rem:reduction_for_N_v2}                                and
  Corollary~\ref{cor:saturated_subgroups}.
\end{proof}

We conclude this section with our  main result showing that the $i$-th
term  of  the  normalizer  chain  is actually  generated  by  the  set
$\mathcal{N}_n^i$ of rigid  commutators defined in Eq.~\eqref{def:Ni}.
We        prove,        indeed,       that        the        inclusion
$\Span{\mathcal{N}^i_{n}}                                          \le
N_{\Sigma_n}\bigl(\Span{\mathcal{N}_n^{i-1}}\bigr)$   shown   in   the
previous proposition is actually an equality.

\begin{theorem}\label{thm:normalizers_indices}
  For $i \leq  n-2$, the group $\Span{\mathcal{N}_n^i}$  is the $i$-th
  term $N_n^{i}$ of the normalizer chain.
\end{theorem}

\begin{proof}
  The  cases  $i=0$  and  $i=1$ has  been  addressed  respectively  in
  Propositions~\ref{prop:U_as_normalizer}   and   \ref{prop:n1}.    We
  assume      by       induction      on      $i\ge       2$      that
  $N_m^{j}=\Span{\mathcal{N}_m^j}$   for   all   $m\le   n$   whenever
  $j<i \le m-2$.  In particular
  \begin{equation*}N^j_m\cap
    \Sigma_{m-1}=N^{j-1}_{m-1}=\Span{\smash{\mathcal{N}_{m-1}^{j-1}}}.
  \end{equation*}
  Notice that
  \begin{equation*}
    \begin{split}
      \Span{\mathcal{N}_n^i}\cap \Sigma_{n-1} &=\Span{\mathcal{N}_{n-1}^{i-1}} \\
      &= N_{n-1}^{i-1}=  N_{\Sigma_{n-1}}(N_{n-1}^{i-2})  \\
      &=N_{\Sigma_{n-1}}(N_{n}^{i-1}\cap \Sigma_{n-1})\\
      &=N_{\Sigma_{n-1}}(N_{n}^{i-1}\cap       \Sigma_{n-1})      \cap
      N_{\Sigma_{n-1}}(N_{n}^{i-1}\cap                           S_n)=
      N_{\Sigma_{n-1}}(N_{n}^{i-1}),
    \end{split}
  \end{equation*}
  where the first equality in the  last line holds since the following
  inclusions
  \begin{equation*}
    [\mathcal{N}_{n-1}^{i-1}, \mathcal{N}_{n}^{i-1}]
    \subseteq  [\mathcal{N}_{n}^{i}, \mathcal{N}_{n}^{i-1}]  \subseteq
    \mathcal{N}_{n}^{i-1}\cup\Set{[\emptyset]}
  \end{equation*}
  imply                                                           that
  $N_{\Sigma_{n-1}}(N_{n}^{i-1}\cap       \Sigma_{n-1})      \subseteq
  N_{\Sigma_{n-1}}(N_{n}^{i-1}\cap  S_n)$.  As  $S_n$  is abelian,  we
  have
  \begin{equation}\label{eq:NNSn}
    \begin{split}
      N^i_n&=N_{\Sigma_n}(N_{n}^{i-1})     =    N_{\Sigma_{n-1}\ltimes
        S_n}(N_{n}^{i-1})                                          \\&
      =N_{\Sigma_{n-1}}(N_{n}^{i-1})N_{S_n}(N_{n}^{i-1})             =
      \Span{\mathcal{N}_{n-1}^{i-1}} N_{S_n}(N_{n}^{i-1}).
    \end{split}
  \end{equation}

  We         are         then        left         to         determine
  $N_{S_n}(N_{n}^{i-1})=          \Set{x\in          S_n          \mid
    [x,\mathcal{N}_{n-1}^{i-1}]\subseteq  N_{n}^{i-1}\cap S_n}$.   Let
  us point  out that,  by Eq.~\eqref{eq:C_lk}  and Eq.~\eqref{def:Ni},
  the groups
  \begin{equation*}
    A\deq   \Span{\mathcal{N}^{i}_n  \cap   S_n}\text{  and   }  B\deq
    \Span{\vphantom{\bigcup}\smash{\bigcup_{j\ge i+1}}\mathcal{W}_{n,\,j+2}}
  \end{equation*} 
  have   trivial   intersection   and  that   $S_n=A\times   B$.    By
  Lemma~\ref{lem:rigid_commutators_normalize}  we have  that $A$  is a
  subgroup  of $N_{n}^{i}\cap  S_n$,  for $1  \le j  \le  i$, so  that
  $N_{n}^{i}=A\times H$ where
  \begin{equation*}
    H\deq  \Big\{x\in  \Span{\vphantom{\bigcup}\smash{\bigcup_{j\ge i+1}}  \mathcal{W}_{n,\,j+2}}
    \,\,\Big|        \,\,        [x,\mathcal{N}_{n-1}^{i-1}]\subseteq
    N_{n}^{i-1}\cap S_n\Big\}.\vphantom{\bigcup_{j\ge i+1}}
  \end{equation*}
  We denote a generic element of $H$ by
  \begin{equation*}
    x\deq \prod_{I\in \mathcal{I}}\puncture{n}{I}^{e_I},
  \end{equation*}
  where the  product is taken  over the  set $\mathcal{I}$ of  all the
  subsets      $I\subseteq      \Set{1,\dots,n-1}$      such      that
  $\sum_{y\in   I}y  \ge   i+3$.   For   $1   \leq  l   \leq  n$   let
  $\mathcal{I}_l=\Set{I\in   \mathcal{I}    \mid   \min(I)=l}$.    Let
  $u=  u_{l,\,  l-1}$ if  $l>  1$,  or  $u=t_1=[1]$ if  $l=1$.   Since
  $x\in H$, we have that
  \begin{equation*}
    [x, u] = \begin{cases} \prod_{I \ni l}\puncture{n}{(I\cup
        \Set{l-1})\setminus   \Set{l}}^{e_I}& \text{if $l>1$}\\
      \prod_{I    \ni   1}\puncture{n}{I\setminus   \Set{1}}^{e_I}& \ \text{if $l=1$}\\
    \end{cases}
  \end{equation*} 
  belongs to  $ \mathcal{N}_{n}^{i-1},$  and in particular  $e_I\ne 0$
  implies
  \begin{equation*}  \sum_{y\in  (I\cup  \Set{l-1})\setminus  \Set{l}}
    y\le i+1.
  \end{equation*}
  If $I\in \mathcal{I}_l$ then
  \begin{equation*}\sum_{y\in (I\cup \Set{l-1})\setminus \Set{l}} y=
    \sum_{y\in I}
    y-l+(l-1)=\sum_{y\in  I}  y-1  \ge  i+2,\end{equation*}  so  that  $e_I=0$  for
  $I\in                       \mathcal{I}_l$.                       As
  $\mathcal{I}=\bigcup_{l=1}^n\mathcal{I}_l$,  we   have  $H=\Set{1}$.
  This  finally  shows that  $N_{S_n}(N^{n}_{i-1})=A=N^n_{i}\cap  S_n =
  \Span{\mathcal{N}_n^i\cap S_n}$
  and, by Eq.~\eqref{eq:NNSn}, 
  that
  \begin{equation*}
    \begin{split}
      N^n_i=N_{\Sigma_n}(N^n_{i-1})&=   \Span{\mathcal{N}_{n-1}^{i-1}}
      N_{S_n}(N_{n}^{i-1}) = \Span{\mathcal{N}_{n-1}^{i-1} \cup
  	(\mathcal{N}_n^i\cap S_n)} \\
      &=     \Span{(\mathcal{N}_{n}^{i}     \cap     \Sigma_{n-1})\cup
        (\mathcal{N}_n^i\cap S_n)} = \Span{\mathcal{N}_n^i},
    \end{split}
  \end{equation*}
  as claimed.
\end{proof}

\subsection{Partitions into at least two distinct parts} This work was
motivated   by   the   computational    evidence   that   the   number
$c_{i}\deq \log_2\Size{N^{i-2}_n : N^{i-3}_n}$ does not depend on $n$,
if $3\le i \le n$ \cite{Aragona2020}.  The first terms of the sequence
$\Set{c_i}$ coincide with  those of the sequence  $\{a_i\}$ defined in
\cite[\url{https://oeis.org/A317910}]{OEIS}, where $a_i$ is the $i$-th
partial sum  of the sequence $\{b_i\}$,  where $b_i$ is the  number of
partitions of $i$ into at least two distinct parts. Some values of the
aforementioned sequences are displayed in Table~\ref{tab:two}.
\begin{table}[hbt]
  {\renewcommand{\arraystretch}{1.3}
    \begin{tabular}{c||c|c|c|c|c|c|c|c|c|c|c|c|c|c|c}
      $i$ & $0$ & $1$ & $2$ & $3$ & $4$ & $5$ & $6$ & $7$ & $8$ & 9& 10 &11&12&13&14\\
      \hline\hline
      ${b_i}$ & 	0& 0& 0& 1& 1& 2& 3& 4& 5& 7& 9& 11& 14& 17& 21\\
      \hline
      $a_i$ &0&0& 0& 1& 2& 4& 7& 11& 16& 23& 32& 43& 57& 74&95
    \end{tabular}
    \bigskip }
  \caption{First values of the sequences $a_i$ and $b_i$}
  \label{tab:two}
\end{table}
\newline  We  have  developed  the rigid  commutator  machinery  as  a
theoretical tool of investigation.  It  is not surprising anymore that
the equality $b_i=\Size{\mathcal{W}_{n,i}}$, where $\mathcal{W}_{n,i}$
is  defined by  Eq.~\eqref{eq:C_lk}, is  the link  with the  mentioned
sequence.   This  combinatorial  identity,  Eq.~\eqref{eq:generators},
Proposition~\ref{prop:size_of_Ni}                                  and
Theorem~\ref{thm:normalizers_indices}  give  at last a  positive   answer  to
Conjecture~1 in \cite{Aragona2020}.
\begin{corollary}
  For $1\le i\le n-2$, the number $\log_{2}\Size{N^{i}_n : N^{i-1}_n}$
  is  independent  of $n$.   It  equals  the  $(i+2)$-th term  of  the
  sequence $\Set{a_j}$ of the partial sums of the sequence $\Set{b_j}$
  counting the number of partitions of  $j$ into at least two distinct
  parts.
\end{corollary}

\section{Normalizers                    of                   saturated
  subgroups}\label{sec:normalizer_saturated}
In    this   section    we    will   prove    that   the    normalizer
$N\deq N_{\Sigma_n}(G)$ of  a saturated subgroup $G$  of $\Sigma_n$ is
also saturated,  provided that $T\le  N$, and thus  we can use  our rigid
commutator machinery  in the  computation of  $N$. In  particular, for
$i\le n-2$,  the machinery  could be  used as  an alternative  tool to
derive    the   theoretical    description    of    $N_n^i$   as    in
Theorem~\ref{thm:normalizers_indices}. Even  if we do not  have such a
description when $i>n-2$, the machinery can be anyway used  to
efficiently compute via $\mathsf{GAP}$ the complete normalizer chain.\newline
\newline
We denote below by $N_i$ the intersection $N\cap S_i$.

\begin{proposition}\label{prop:normalizer_homogeneous}
  If   $G$    is   a    saturated   subgroup   of    $\Sigma_n$,   and
  $N=N_{\Sigma_n}(G)$ is its normalizer in $\Sigma_n$, then
  \begin{equation*}
    N= N_1 \ltimes  \cdots \ltimes N_n= \ltimes_{i=1}^n
    N_{S_i}(G).
  \end{equation*}
  In particular if  $x\in N$ and $x=x_1\cdots x_n$,  with $x_i\in S_i$
  for all $1 \leq i \leq n$, then $x_i\in N_i$ for all~$1 \leq i \leq n$.
\end{proposition}

\begin{proof}
  Let   $x\in   N$   and    write   $x=x_{i_1}\dots   x_{i_k}$   where
  $   1   \le  i_1   <   \dots   <   i_k   \le  i_{k+1}\deq   n$   and
  $x_ {i_j}\in S_{i_j}$, for $1 \leq j \leq k$.  In order to prove our
  claim we  first show that  $[x_{i_1},c]\in G$ for  every non-trivial
  rigid commutator  $c$ of  $G$.  Since  $G$ is  generated by  its own
  non-trivial rigid  commutators, it will follow  that $x_{i_1}\in N$.   As a
  consequence, also $x_{i_2}\dots  x_{i_k} \in N$. Thus,  we may argue
  by  induction  on  $k$  to   obtain  that  $x_{i_j}\in  N$  for  all
  $1\leq j \leq k$.
	
  Let  $i$  be  such  that  $c\in  G\cap  S_i$.   Suppose  first  that
  $i  <  i_1$.  If  $[c,x_{i_1}]=  1$,  then $[c,x_{i_1}]\in  G$.   If
  $[c,x_{i_1}]\ne   1$,   then    $[c,x]=[c,x_{i_1}]h\in   G$,   where
  $[c,x_{i_1}]\in   S_{i_1}$   and    $h\in   \prod_{t>i_1}S_t$.    By
  Corollary~\ref{cor:homogeneus_components2}     we    obtain     that
  $[c,x_{i_1}]\in  G\cap S_{i_1}\le  G$.   If  $i=i_1$, then  trivially
  $[c,x_{i_1}]=1\in   G$.    The   last   possibility  is
  $i_1 <\dots <  i_m < i\le i_{m+1}$ for some  $m\le k$.  Suppose that
  $[x_{i_1},c]\ne 1$. In this case
  \begin{align*}
    G\ni [x,c] &= [x_{i_1}\dots x_{i_k},c] = [x_{i_1},c]^{x_{i_2}\dots x_{i_k}} \cdot [x_{i_2}\dots x_{i_k},c] \\
               &= \bigl([x_ {i_1},c]^{x_ {i_2}\dots x_ {i_m}}\bigr)^{x_ {i_{m+1}}\dots x_ {i_k}} \cdot  [x_ {i_2}\dots x_ {i_k},c]\\
               &=[x_ {i_1},c]^{x_ {i_2}\cdots x_ {i_m}} \cdot [[x_ {i_1},c]^{x_ {i_2}\cdots x_ {i_m}},x_ {i_{m+1}}\cdots x_ {i_k}] \cdot  [x_ {i_2}\dots x_ {i_k},c].
  \end{align*}
  Let us consider the commutators
  \begin{align*}
    [x_ {i_1},c]^{x_ {i_2}\cdots x_ {i_m}}&=[x_ {i_1},c][[x_ {i_1},c], x_ {i_2}\cdots x_ {i_m}]=d_1\cdots d_t,\\
    [[x_ {i_1},c]^{x_ {i_2}\cdots x_ {i_m}},x_ {i_{m+1}}\cdots x_ {i_k}]\ &= m_1\cdots m_r,\\
    [x_ {i_2} \cdots x_ {i_m}\cdot x_ {i_{m+1}} \cdots x_ {i_k},c]&=f_1\cdots f_s \cdot l_1\cdots l_u,
  \end{align*}
  written as ordered product of distinct rigid commutators
  \begin{equation*}
    d_1,\dots,  d_t,  f_1,\dots,  f_s  \in
    G\cap S_{i},
  \end{equation*}
  and
  \begin{equation*}
    m_1,\ldots,   m_r,  l_1,\dots,   l_u  \in   G  \cap
    (S_{i+1}\ltimes \cdots \ltimes  S_{n}).
  \end{equation*}
  Notice                                                          that
  $\Set{d_1,\dots,  d_t} \cap  \Set{ f_1,\dots,  f_s}=\emptyset$ since
  the commutators  $d_i$ are of the  form $[X]$ for some  set $X$ with
  $i_1\in X$, whereas the commutators $f_j$  are of the form $[Y]$ for
  some    set     $Y$    with    $i_1\notin    Y$.      This    yields
  $[x_  {i_1},c]^{x_  {i_2}\cdots x_  {i_m}}  \in  G\cap S_i$  and  so
  $[x_ {i_1},c] \in G\cap S_i \le G$.
\end{proof}

\begin{lemma}\label{lem:normalizer_rigidc_components}
  Suppose that $G$ is a saturated subgroup of $\Sigma_n$ normalized by
  $T$. If $x_1,\dots, x_k\in S_j$  are distinct rigid commutators such
  that   $x=x_1\cdots   x_k\in   N$,   then   $x_i\in   N$   for   all
  $1\leq i \leq k$.
\end{lemma}

\begin{proof}
  Let     $c_1,\ldots,    c_h\in     \mathcal    R^*$     such    that
  $G=\Span{c_1,\ldots, c_h}$ and let us write every $c_s$ and $x_t$ in
  punctured         form:        $c_s=\puncture{m_s}{C_s}$         and
  $x_t=\puncture{j}{X_t}$.
	
  Suppose first that $m_s <j$, so that
  \begin{equation}\label{eq:expansion}
    [c_s,x]=\prod_{t=1}^k d_{s,t} \in G\cap S_j,
  \end{equation}
  where
  $d_{s,t}\deq[c_s,x_t]    =     \puncture{j}{C_s\cup    (X_t\setminus
    \Set{m_s})} $.   Notice that  if the commutator  $d_{s,t}$ appears
  only       once      in       the       product,      then,       by
  Corollary~\ref{cor:homogeneus_components},   $d_{s,t}\in   G$.    If
  $C_s\cap  X_t=\emptyset$ for  all $1\leq  t  \leq k$,  then all  the
  non-trivial  $d_{s,t}$ appearing  in  the product  are distinct  and
  hence they appear  only once in the product, so  that $d_{s,t}\in G$
  for all $1 \leq t \leq  k$.  If $C_s\cap X_t\ne \emptyset$, then the
  commutator  $d_{s,t}$  may appear  more  than  once in  the  product
  displayed in  Eq.~\eqref{eq:expansion}.  Let $l\in C_s\cap  X_t$ and
  consider                        the                       commutator
  $c_{s,l}=[c_s, t_l] =  \puncture{m_s}{C_s\setminus\Set{l}} \in G$ as
  $t_l=[l,\dots,1]\in T\le N_{\Sigma_n}(G)$.  Notice that
  \begin{equation*}[    c_{s,l},x_t]   =    \puncture{j}{(C_s\setminus
      \Set{l})\cup  (X_t\setminus   \Set{m_s})}=  \puncture{j}{C_s\cup
      (X_t\setminus \Set{m_s})} =[c_{s},x_t]=d_{s,t}.
  \end{equation*}
  Let  $C=C_s\setminus\Set{l}$.   We  have   determined  a  new  rigid
  commutator    $c=c_{s,l}=\puncture{m_s}{C}\in     G$    such    that
  $\Size{C\cap      X_t}     <      \Size{C_s\cap     X_t}$,      that
  $\Size{C} < \Size{C_s}$  and that $ d_{s,t}=[c,x_t]$  appears in the
  expansion of $[c,x]$. Using the same strategy, after a finite number
  of   steps,  we   obtain   $c=\puncture{m_s}{C}\in   G$  such   that
  $C\cap X_t=\emptyset$.  If $ d_{s,t}=[c,x_t]=[c,x_{t_1}]=d_{s,t_1}$,
  for  some  $t_1\ne  t$,  then $C\cap  X_{t_1}\ne  \emptyset$,  since
  otherwise   $X_t=X_{t_1}$   and  consequently   $x_t=x_{t_1}$   with
  $t\ne t_1$, contrary  to the hypotheses. Thus we may  proceed in the
  same way with $d_{s,t_1}$. Since at each step the cardinality of $C$
  is strictly decreasing, after a finite number
  of   steps  we   find   a   $c\in  G$   and   $x_{t_r}$  such   that
  $d_{s,t}=d_{s,t_1}=\dots =  d_{s,t_r}$ appears only once  in $[c,x]$
  giving $d_{s,t}\in G$.  This finally shows that $d_{s,t}\in G$
  for all $1\leq t \leq k$.\newline
  \newline
  If $j=m_s$ then $x_i$ and $c_s$ commute for all $i$ and there is nothing to prove.\newline
  \newline
  We are left with the case when $m_s>j$. As above, we have
  \begin{equation*}
    [c_s,x]=\prod_{t=1}^k d_{s,t} \in G\cap S_j,
  \end{equation*}
  where
  $d_{s,t}\deq[c_s,x_t] = \puncture{m_s}{(C_s  \setminus \Set{j}) \cup
    X_t} $.  Reasoning as we did for $m_s<j$, we
  obtain that $d_{s,t}\in G$ for all $1\leq t \leq k$.
	
  In  all  the   cases  we  have  proved  that  $x_i\in   N$  for  all
  $1\leq i \leq k$, which is our claim.
\end{proof}
As            an             easy            consequence            of
Proposition~\ref{prop:normalizer_homogeneous}                      and
Lemma~\ref{lem:normalizer_rigidc_components}  we  find  the  following
result.
\begin{theorem}\label{thm:normalizer_saturated}
  The  normalizer  $N$  in  $\Sigma_n$  of  a  saturated  subgroup  of
  $\Sigma_n$ is also saturated, provided that $N$ contains $T$.
\end{theorem}

\begin{remark}\label{rem:normal_izer_closure}
  Let $\mathcal{A}$ and $\mathcal{B}$  be two subsets of $\mathcal{R}$
  such                                                            that
  $\Set{t_1,\dots, t_n} \subseteq  \mathcal{A} \subseteq \mathcal{B}$,
  let   $A=\Span{\mathcal{A}}$  and   $B=\Span{\mathcal{B}}$  be   the
  corresponding  saturated subgroups.   It is  easy to  recognize that
  $N_B(A)=\Span{b\in    \mathcal{B}   \mid    [b,\mathcal{A}]\subseteq
    \mathcal{A}\cup   \Set{[\emptyset]}}$.    Similarly,  the   normal
  closure  $A^B$ of  $A$  in  $B$ is  the  subgroup  generated by  the
  intersection of all the  subsets $\mathcal{C}$ of $\mathcal{R}$ such
  that  $\mathcal{A} \subseteq  \mathcal{C} \subseteq\mathcal{B}$  and
  $[\mathcal{C}    ,\mathcal{B}   ]    \subseteq   \mathcal{C}    \cup
  \Set{[\emptyset]}$.  In particular, both the normalizer $N_B(A)$ and
  the normal closure $A^B$ are saturated.
\end{remark}
\begin{remark}
  The   condition   that   $T$   is  contained   in   the   normalizer
  $N=N_{\Sigma_n}(G)$ of  a saturated  subgroup $G$ cannot  be removed
  from   the hypotheses of Theorem~\ref{thm:normalizer_saturated}.      Indeed,     if
  $G=\Span{[n,\ldots,3]}$,  then  the   product  $[2]\cdot  [2,1]$  is
  contained  in  the  centralizer  of  $A$,  and  hence  also  in  the
  normalizer $N$ of  $A$, but none of the two  rigid commutators $[2]$
  or   $[2,1]   \in   T$    normalizes   $G$.    In   particular,   by
  Corollary~\ref{cor:homogeneus_components2}, the  subgroup $N$ cannot
  be saturated.
\end{remark}

\begin{remark}\label{rmk_algo}
  Another  proof   of  Theorem~\ref{thm:normalizers_indices}   can  be
  obtained  by  Theorem~\ref{thm:normalizer_saturated}.    Indeed, it  is  not 
  difficult, but rather tedious, to check that
  \begin{equation*}\mathcal{N}_{n}^{i}=\Set{c\in \mathcal{R}^* \mid
      [c,\mathcal{N}_{n}^{i-1}]\subseteq   \mathcal{N}_{n}^{i-1}  \cup
      \Set{[\emptyset]}}.\end{equation*}  for $0\le  i  \le n-2$.
  The result  then
  follows        by       Proposition~\ref{prop:size_of_Ni}.
\end{remark}

From                            Theorems~\ref{thm:normalizers_indices}
and~\ref{thm:normalizer_saturated}  and from  Remark~\ref{rmk_algo} we
derive  a straightforward  corollary resulting  in an  algorithm whose
$\mathsf{GAP}$    implementation    is     publicly    available    at
\href{https://github.com/ngunivaq/normalizer-chain}{\textsf{GitHub}}\footnote{\
  See \url{https://github.com/ngunivaq/normalizer-chain}}. This script
allows a significant speed-up in the computation of the normalizer $N$
of  a saturated  subgroup provided  that $N$  contains $T$.   We could
easily apply  this script to  compute our  normalizer chain up  to the
dimension $n=22$. For example, whereas the standard libraries required
one month on a cluster to compute the terms of the normalizer chain in
$\Sym(2^{10})$, our  implementation of the rigid  commutator machinery
gives the result  in a few minutes,  even on a standalone  PC.  With a
similar approach,  we can  also use rigid  commutators to  compute the
normal closure  of a  saturated subgroup.  Some  explicit calculations
are    shown    below     in    Section~\ref{sec:computation}.     Let
$\mathcal{M}_n^i$ be the set of all the rigid commutators belonging to
$N_n^i$.   From Theorem~\ref{thm:normalizer_saturated},  the subgroups
$N_n^i$ are saturated, hence  $N_n^i = \Span{\mathcal{M}_n^i}$ for all
$i \geq 1$.

\begin{corollary}\label{algo}
  The  set $\mathcal{M}_n^i$  is the  largest subset  of $\mathcal{R}$
  that normalizes $\mathcal{M}_n^{i-1}$, i.e.\
  \begin{equation*}\mathcal{M}_n^i=
    \Set{c\in   \mathcal{R}  \mid   [c,\mathcal{M}_n^{i-1}]  \subseteq
      \mathcal{M}_n^{i-1}}.\end{equation*}                                Moreover,
  $\mathcal{N}_n^i  =\mathcal{M}_n^i \setminus  \Set{[\emptyset]}$ for
  $1\le i\le n-2$.
\end{corollary}
The construction of the terms of  the normalizer chain is then reduced
to the determination of the sets $\mathcal{M}_n^i$, a task which turns
out to be way faster than computing the terms of the normalizer chains
as  subgroups  of  $\Sigma_n$   via  the  \texttt{Normalizer}  command
provided by \textsf{GAP}.

	\section{A computational supplement}\label{sec:computation}
In this  section we show  an explicit  construction of the  first four
groups in the normalizer chain when $n=6$, i.e.\ in $\Sym(64)$. Let us
start  with  determining the  generators  of  $T$  in terms  of  rigid
commutators:
\begin{eqnarray*}
	t_1 =& [1] &= (1, 33)(2, 34)(3, 35) \ldots (30, 62)(31, 63)(32, 64),\\
	t_2 =& [2,1] &= (1, 17)(2, 18)(3, 19) \ldots (46, 62)(47, 63)(48, 64),\\
	t_3 = &[3,2,1] &= (1, 9)(2, 10)(3, 11) \ldots (54, 62)(55, 63)(56, 64), \\
	t_4 = &[4,3,2,1] &=(1, 5)(2, 6)(3, 7) \ldots (58, 62)(59, 63)(60, 64), \\
	t_5 = &[5,4,3,2,1] &= (1, 3)(2, 4)(5, 7) \ldots (58, 60)(61, 63)(62, 64),\\
	t_6 =& [6,5,4,3,2,1] &= (1, 2)(3, 4)(5, 6) \ldots (59, 60)(61, 62)(63, 64).
\end{eqnarray*}
We  have   that  $T  =   \Span{t_1,  t_2,  \ldots,  t_6}$   and,  from
Proposition~\ref{prop:U_as_normalizer},  its normalizer  in $\Sigma_n$
is
$N_6^0 = U_6=\Span{\mathcal U_6} = \Span{T,u_{ij} \mid 1 \le j < i \le
	6}$. Thus the generators of $N^0_6$, besides those of $T$, are
\begin{eqnarray*}
	\puncture{6}{5}, \puncture{6}{4}, \puncture{6}{3}, \puncture{6}{2}, \puncture{6}{1}, \\
	\puncture{5}{4}, \puncture{5}{3}, \puncture{5}{2}, \puncture{5}{1},  \\
	\puncture{4}{3}, \puncture{4}{2}, \puncture{4}{1},  \\
	\puncture{3}{2}, \puncture{3}{1},  \\
	\puncture{2}{1},
\end{eqnarray*}
consequently   $|N_6^0|   =   2^{21}$.   Now,   in   accordance   with
Eq.~\eqref{def:Ni}   and  Theorem~\ref{thm:normalizers_indices},   the
normalizer $N_6^1$  is generated  by the rigid  commutators previously
listed and by  $\eta_6$, the only element  of $\mathcal{W}_{6,3}$ (see
Eq.~\eqref{eq:C_lk}). The  commutator $\eta_6$ is the  punctured rigid
commutator based at $6$ and missing the integers $1$ and $2$, i.e.
\begin{equation}\label{eq:N1}
	\eta_6=[6,5,4,3] = \puncture{6}{2,1},
\end{equation}
where $1$ and  $2$ indeed represent the sole partition  of $3$ into at
least        two       distinct        parts.        From        this,
$\log_2\Size{N^{1}_6   :   N^{0}_6   }=   1  =   a_3$.    Again   from
Eq.~\eqref{def:Ni}   and  Theorem~\ref{thm:normalizers_indices},   the
normalizer  $N_6^2$  is generated,  along  with  the elements  already
mentioned,  by  the  rigid   commutators  in  $\mathcal{W}_{5,3}$  and
$\mathcal{W}_{6,4}$, i.e.\
\begin{eqnarray}
	\label{def_N2_1} [5,4,3] &=& \puncture{5}{2,1},\\[1pt]
	\label{def_N2_2} [6,5,4,2]&=&\puncture{6}{3,1}.
\end{eqnarray}
The   commutator    of   Eq.~\eqref{def_N2_1},   which    belongs   to
$\mathcal{W}_{5,3}$, is  the punctured  rigid commutator based  at $5$
and   missing  the   integers  $1$   and  $2$.    The  commutator   of
Eq.~\eqref{def_N2_2} instead, which belongs to $\mathcal{W}_{6,4}$, is
based at $6$  and misses the integers $1$ and  $3$, composing the sole
partition  of $4$  into  at  least two  distinct  parts.  Notice  that
$[5,4,3]   =    [\cancel{6},5,4,3]$.    Indeed,   as    discussed   in
Fact.~\ref{rem:reduction_for_N},        the       commutator        of
Eq.~\eqref{def_N2_1}   is    obtained   by   lifting   the    one   of
Eq.~\eqref{eq:N1},  i.e.\ by  removing $6$,  the left-most  element of
$\eta_6$.   We  have  $\log_2\Size{N^{2}_6  : N^{1}_6  }=  2  =  a_4$.
Similarly, $N_6^3$ is generated by adding the new rigid commutators
\begin{eqnarray}
	\label{def_N3_1}  [\cancel{5},4,3] = & [4,3] &= \puncture{4}{2,1},\\[1pt]
	\label{def_N3_2} [\cancel{6},5,4,2]  = & [5,4,2] &=\puncture{5}{3,1},\\[1pt]
	&\label{def_N3_3} [6,5,3,2] &= \puncture{6}{4,1},\\[1pt]
	&\label{def_N3_4} [6,5,4,1]&=\puncture{6}{3,2},
\end{eqnarray}
where the commutators of Eq.~\eqref{def_N3_1} and Eq.~\eqref{def_N3_2}
are respectively obtained by lifting those of Eq.~\eqref{def_N2_1} and
Eq.~\eqref{def_N2_2}, and the  commutators of Eq.~\eqref{def_N3_3} and
Eq.~\eqref{def_N3_4}   belong  to   $\mathcal{W}_{6,5}$,  respectively
corresponding  to the  partitions $4+1$  and  $3+2$ of  $5$.  At  this
stage,  we have  that  $\log_2\Size{N^{3}_6  : N^{2}_6  }=  4 =  a_5$.
Ultimately, the commutators
\begin{eqnarray*}
	[\cancel{4},3] = & [3] &= \puncture{3}{2,1},\\[1pt]
	[\cancel{5},4,2]  = & [4,2] &=\puncture{4}{3,1},\\[1pt]
	[\cancel{6},5,3,2] = & [5,3,2] &= \puncture{5}{4,1},\\[1pt]
	[\cancel{6},5,4,1] = & [5,4,1]&=\puncture{5}{3,2},\\[1pt]
	&[6,4,3,2]&=\puncture{6}{5,1},\\[1pt]
	&[6,5,3,1]&=\puncture{6}{4,2},\\[1pt]
	&[6,5,4]&=\puncture{6}{3,2,1}
\end{eqnarray*}
complete    the   set    of    rigid    generators of   $N_6^4$,    and
$\log_2\Size{N^{4}_6 : N^{3}_6 }= 7 = a_6$.\newline
\newline
Using  Corollary~\ref{algo}, we  can  find a  saturated  set of  rigid
generators  for  all the  elements  of  the  chain.  Notice  that  for
$i > 5$, the sequence $\log_2\Size{N^{i}_6 : N^{i-1}_6 }$ does not fit
the  pattern of  the sequence  $\{a_j\}$. Although  we do  not have  a
general  formula  to calculate  the  values  of the  relative  indices
between two  consecutive terms  in the normalizer  chain, they  can be
explicitly       determined       by      the       algorithm       in
\href{https://github.com/ngunivaq/normalizer-chain}{\textsf{GitHub}}.
Computational results  are summarized in Table~\ref{tab:n6},  where we
list all the relative indices of  the normalizer chain.  In the second
column, the logarithms of the sizes  of the intersections of each term
with each of the subgroups $S_6, \ldots, S_1$ are displayed.

\begin{table}[hptb]
	\begin{tabular}{c||c|c|c}
		$i$ & $
		\begin{matrix}
			\dim(N_6^i\cap S_j) \\
			j=6,5,4,3,2,1
		\end{matrix}
		$ & $\log_2\Size{N_6^i}$ & $\log_2\Size{N_6^i : N_6^{i-1}}$ \\
		\hline\hline $0$ &  6, 5, 4, 3, 2, 1  & $21 $ & $15$ \\ 
		\hline $1$ &  7, 5, 4, 3, 2, 1  & $22 $ & $1$ \\ 
		\hline $2$ &  8, 6, 4, 3, 2, 1  & $24 $ & $2$ \\ 
		\hline $3$ &  10, 7, 5, 3, 2, 1  & $28 $ & $4$ \\ 
		\hline $4$ &  13, 9, 6, 4, 2, 1  & $35 $ & $7$ \\ 
		\hline $5$ &  14, 10, 6, 4, 2, 1  & $37 $ & $2$ \\ 
		\hline $6$ &  16, 11, 7, 4, 2, 1  & $41 $ & $4$ \\ 
		\hline $7$ &  18, 12, 8, 4, 2, 1  & $45 $ & $4$ \\ 
		\hline $8$ &  19, 12, 8, 4, 2, 1  & $46 $ & $1$ \\ 
		\hline $9$ &  20, 12, 8, 4, 2, 1  & $47 $ & $1$ \\ 
		\hline $10$ &  21, 13, 8, 4, 2, 1  & $49 $ & $2$ \\ 
		\hline $11$ &  22, 14, 8, 4, 2, 1  & $51 $ & $2$ \\ 
		\hline $12$ &  23, 15, 8, 4, 2, 1  & $53 $ & $2$ \\ 
		\hline $13$ &  24, 16, 8, 4, 2, 1  & $55 $ & $2$ \\ 
		\hline $14$ &  25, 16, 8, 4, 2, 1  & $56 $ & $1$ \\ 
		\hline $15$ &  26, 16, 8, 4, 2, 1  & $57 $ & $1$ \\ 
		\hline $16$ &  27, 16, 8, 4, 2, 1  & $58 $ & $1$ \\ 
		\hline $17$ &  28, 16, 8, 4, 2, 1  & $59 $ & $1$ \\ 
		\hline $18$ &  29, 16, 8, 4, 2, 1  & $60 $ & $1$ \\ 
		\hline $19$ &  30, 16, 8, 4, 2, 1  & $61 $ & $1$ \\ 
		\hline $20$ &  31, 16, 8, 4, 2, 1  & $62 $ & $1$ \\ 
		\hline $21$ &  32, 16, 8, 4, 2, 1  & $63 $ & $1$ \\ 
	\end{tabular}
	\bigskip
	\caption{The normalizer chain for $n=6$}
	\label{tab:n6}
\end{table}

\section{Problems for future research}\label{sec:concl}
We  conclude this  work by highlighting  some further properties and structures of  the set
$\mathcal R$ of  rigid commutators  and providing  some hints  for future
research. 
\subsection*{Algebras of rigid commutators}
The  operation  of commutation  in  $\mathcal{R}$  is commutative  and
$[\emptyset]$ represents the \emph{zero} element.  Moreover, for every
$x,y \in \mathcal{R}$ the following identity is satisfied
\begin{equation*}
  [[x,x,y],x] = [[x,x],[y,x]] \text{\ \ \ \ \ \ \ {(Jordan
      identity)}}.
\end{equation*}
Let $\F$ be any field of  characteristic $2$ and let $\mathfrak{r}$ be
the  vector space  over $\F$  having  the set  $\mathcal{R}^*$ of  the
non-trivial rigid commutators as a basis.  The space $\mathfrak{r}$ is
endowed  with  a  natural  structure   of  an  algebra.   The  product
$x\star y$ of two rigid commutators $x,y \in \mathcal R$ is defined as
\begin{equation*}
  x\star y \deq
  \begin{cases}
    [x,y] & \text{if $[x,y]\ne [\emptyset]$}\\
    0 &\text{otherwise.}
  \end{cases}
\end{equation*}
This  operation  is  then  extended to  the  whole  $\mathfrak{r}$  by
bilinearity  and turns  $\mathfrak{r}$ into  a \emph{Jordan  algebra},
since it is commutative and $x\star x=0$ for all $x \in \mathfrak{r}$.
Moreover, if  $\mathcal{H}$ is a saturated  subset of $\mathcal{R}^*$,
then, on the one hand  the group $H=\Span{\mathcal{H}}$ is a saturated
subgroup of  $\Sigma_n$ and, on  the other hand, the  $\F$-linear span
$\mathfrak{h}$ of $\mathcal{H}$ is a subalgebra of $\mathfrak{r}$. The
property
$[\mathcal{R},\mathcal{H}]          \subseteq          \mathcal{H}\cup
\Set{[\emptyset]}$ is a necessary and  sufficient condition for $H$ to
be a  normal subgroup of  $\Sigma_n$ and  for $\mathfrak{h}$ to  be an
ideal  of   $\mathfrak{r}$.   We   point  out   that  the   fact  that
$\mathcal{R}$  is closed  under commutation  is crucial  to check  the
previous statement.   If $c$  is the  nilpotency class  of $\Sigma_n$,
then the product  of $c+1$ elements of $\mathfrak{r}$  is always zero,
so that $\mathfrak{r}$ is nilpotent.   The study of the properties and
the  representations  of  this  algebra  seems  to  be  a  problem  of
independent interest,  in connection with  the study of  the saturated
subgroups of $\Sigma_n$.

\subsection*{Again on the normalizer chain}
We   have  obtained   from  Theorem~\ref{thm:normalizers_indices}   an
explicit description of the non-trivial rigid generators of the $i$-th
term of the normalizer  chain when $1 \leq i \leq  n-2$, i.e.\ the set
$  \mathcal{N}_n^i$.   We have seen  that  $\mathcal{N}_n^i$  has a  nice
description  by  way  of Eq.~\eqref{eq:C_lk}  and  Eq.~\eqref{def:Ni},
i.e.\ it  is generated by  some rigid commutators either  belonging to
$\mathcal{U}_n$ or  having a punctured form  corresponding to suitable
partitions into at least two distinct parts.
Although we can efficiently compute  all the normalizers in the chain,
as described  in the  lines following  Corollary~\ref{algo}, it  is an
interesting problem  to find a  similar combinatorial formula  for the
generating set  $\mathcal{M}_n^i$ of $N^i_n$ when  $i>n-2$.  Moreover,
as already  mentioned in Section~\ref{sec:computation}, the  values of
the sequence  $\log_2\lvert N_n^i  : N_n^{i-1}\rvert$  do not  seem to
belong to any special known pattern when $i > n-2$.  Table~\ref{tab:4}
contains  the values  of  $\log_2\lvert N_n^i  : N_n^{i-1}\rvert$  for
$1\le i\le  14$ and  $3 \leq  n \leq 15$.   It is  an open  problem to
determine the general behavior of the sequence.

\begin{table}[phbt]
  \label{tab:nindices}
  \begin{tabular}{c||c|c|c|c|c|c|c|c|c|c|c|c|c|c}
    $n$ & \multicolumn{14}{c}{  $\vphantom{\Big|} \log_2\lvert N_n^i : N_n^{i-1}\rvert$ for $1 \leq i \leq 14$}                                                \\ \hline \hline 
    3  &  \hl{1} & 0& 0& 0& 0& 0& 0& 0& 0& 0& 0& 0& 0& 0   \\ \hline 
    4 &  \hl{1} & \hl{2}& 1& 1& 0& 0& 0& 0& 0& 0& 0& 0& 0& 0   \\ \hline 
    5 &    \hl{1}& \hl{2}& \hl{4}& 1& 2& 2& 1& 1& 1& 1& 0& 0& 0& 0  \\ \hline 
    6 &   \hl{1}& \hl{2}& \hl{4}& \hl{7}& 2& 4& 4& 1& 1& 2& 2& 2& 2& 1  \\ \hline 
    7  &   \hl{1}& \hl{2}& \hl{4}& \hl{7}& \hl{11}& 4& 7& 3& 4& 2& 2& 4& 4& 4  \\ \hline 
    8  &  \hl{1}& \hl{2}& \hl{4}& \hl{7}& \hl{11}& \hl{16}& 7& 5& 6& 2& 6& 6& 3& 3   \\ \hline 
    9  &   \hl{1}& \hl{2}& \hl{4}& \hl{7}& \hl{11}& \hl{16}& \hl{23}& 4& 9& 4& 11& 4& 12& 9   \\ \hline 
    10  &   \hl{1}& \hl{2}& \hl{4}& \hl{7}& \hl{11}& \hl{16}& \hl{23}& \hl{32}& 4& 14& 5& 20& 7& 19   \\ \hline 
    11 &  \hl{1}& \hl{2}& \hl{4}& \hl{7}& \hl{11}& \hl{16}& \hl{23}& \hl{32}& \hl{43}& 5& 22& 7& 32& 4   \\ \hline 
    12 &   \hl{1}& \hl{2}& \hl{4}& \hl{7}& \hl{11}& \hl{16}& \hl{23}& \hl{32}& \hl{43}& \hl{57}& 7& 32& 12& 43   \\ \hline 
    13 &   \hl{1}& \hl{2}& \hl{4}& \hl{7}& \hl{11}& \hl{16}& \hl{23}& \hl{32}& \hl{43}& \hl{57}& \hl{74}& 12& 42& 18   \\ \hline 
    14 &  \hl{1}& \hl{2}& \hl{4}& \hl{7}& \hl{11}& \hl{16}& \hl{23}& \hl{32}& \hl{43}& \hl{57}& \hl{74}&\hl{95}& 8& 24   \\ \hline 
    15 &   \hl{1}& \hl{2}& \hl{4}& \hl{7}& \hl{11}& \hl{16}& \hl{23}& \hl{32}& \hl{43}& \hl{57}& \hl{74}& \hl{95}& \hl{121}& 8   \\ 
  \end{tabular}
  \bigskip
  \caption{\footnotesize                   Values                   of
    $\log_2\lvert N_n^i : N_n^{i-1}\rvert$ for  small $i$ and $n$. For
    $i\le n-2$ these numbers do not depend on $n$ and in the table are
    represented by highlighted digits.}
  \label{tab:4}
\end{table}

\subsection*{An \emph{odd} generalization}
It appears natural  to ask whether a similar  rigid commutator machinery
can be developed  in  a  Sylow $p$-subgroup  of  the  symmetric  group
$\Sym(p^n)$  when $p$  is an  odd prime.   This looks  as an  entirely
different problem in  terms of techniques and results. For example, a
rigid commutator could contain  repetitions. Although such a machinery
might  have  a   weaker  cryptographic  application,  it   may  turn  out
interesting on a computational point of view.

 \section*{Acknowledgment}

 We  thank  the  staff  of  the \emph{Department  of  Information
 Engineering, Computer  Science and  Mathematics} at the  University of
 L'Aquila for helping us in managing the HPC cluster CALIBAN, which we
 extensively       used       to       run       our       simulations
 (\url{caliband.disim.univaq.it}).
  We   are  also  grateful   to  the
 \emph{Istituto  Nazionale   d'Alta  Matematica  -  F.\   Severi}  for
 regularly  hosting our  research seminar  \emph{Gruppi al  Centro} in
 which this paper was conceived.
 
\bibliographystyle{amsalpha}
\bibliography{sym2n_ref}

\end{document}